\documentclass[11pt]{amsart}

\usepackage{graphicx}
\usepackage[utf8]{inputenc}
\usepackage[english]{babel}
\usepackage{amsmath} 
\usepackage{amssymb}
\usepackage{amsthm}
\usepackage{mathtools}
\usepackage{graphicx}
\graphicspath{ {c:/This PC/Documents/images/} }
\usepackage{tikz-cd}
\usepackage{MnSymbol}
\usepackage{amsthm}
\usepackage{tikz-cd}
\usepackage{extarrows}
\usepackage[top=30mm, left=35mm, right=35mm]{geometry}
\usepackage{csquotes}

\usepackage[backend=biber, style=alphabetic, sorting=nyt]{biblatex}
\AtBeginBibliography{\footnotesize}

\theoremstyle{plain}
\newtheorem{theorem}{Theorem}[section]
\newtheorem*{theorem*}{Theorem}

\newtheorem{lemma}[theorem]{Lemma}

\newtheorem{corollary}[theorem]{Corollary}
\newtheorem*{corollary*}{Corollary}

\theoremstyle{definition}
\newtheorem{definition}[theorem]{Definition}

\newtheorem*{remark}{Remark}

\newtheorem{example}[theorem]{Example}

\newtheorem{innercustomthm}{Corollary}
\newenvironment{customthm}[1]
  {\renewcommand\theinnercustomthm{#1}\innercustomthm}
  {\endinnercustomthm}

\newcommand{\Mod}{\text{Mod}}

\title{Comparing Teichm\"uller and curve graph translation lengths}
\author{Philipp Bader}
\date{}

\begin{document}

\maketitle

\begin{abstract}
    A pseudo-Anosov mapping class acts on Teichm\"uller space $\mathcal{T}$ as well as on the curve graph $\mathcal{C}$ with so called north-south dynamics. We can measure a stable translation length $l_\mathcal{T}$ and $l_\mathcal{C}$ of the respective actions. Boissy--Lanneau compute the minimal Teichm\"uller translation length over all pseudo Anosovs in a fixed genus that lie in a hyperelliptic component of translation surfaces. In particular, this minimum is always greater than $\log(\sqrt{2}),$ independently of the genus. Here, we show that the minimal stable curve graph translation length over the same family of pseudo-Anosovs behaves differently: Namely, for a genus $g$ surface this minimal translation length is of order $\frac{1}{g}.$ To prove this result, we combine techniques that are used to find upper and lower bounds for the stable curve graph translation length with the Rauzy-Veech induction machinery.\\
    \indent We proceed with showing that for fixed genus $g$ there is a sequence of pseudo-Anosovs $f_n$ with 
    $$\lim\limits_{n \to \infty} l_\mathcal{T}(f_n) = \infty \text{ and } l_\mathcal{C}(f_n) \le \frac{1}{g-1}$$ for all $n \in \mathbb{N}.$ As a corollary, we obtain that there are stable curve graph translation lengths with infinite multiplicity, i.e. there exists $q \in \mathbb{Q}$ and infinitely many, non-conjugate pseudo-Anosovs $f_n$ with $l_\mathcal{C}(f_n) = q$ for all $n.$\\
    \indent Finally, for a pseudo-Anosov $f$ -satisfying some technical conditions- we present a method to construct infinitely many, non-conjugate pseudo-Anosovs with the same stable curve graph translation length as $f.$ 
\end{abstract}

\tableofcontents

\section{Introduction}

Let $S$ be a closed, oriented surface of genus $g \ge 2.$ The \textit{mapping class group} $\Mod(S)$ is the group of orientation preserving homeomorphisms up to isotopy. $\Mod(S)$ admits natural actions on the \textit{Teichm\"uller space} $\mathcal{T}(S)$ and on the \textit{curve graph} $\mathcal{C}(S).$ The former is the space of marked hyperbolic structures on $S,$ while the latter is a graph whose vertices correspond to isotopy classes of essential, simple, closed curves and its edges to disjointness of these curves. $\mathcal{T}(S)$ admits a metric called the \textit{Teichm\"uller metric} which resembles the hyperbolic metric in many properties. For example, the isometries of $\mathcal{T}(S)$ can be split into three categories, just like isometries of hyperbolic space are classified as being elliptic, parabolic or hyperbolic. The action of the mapping class group on Teichm\"uller space was used by Thurston to categorize mapping classes. The Nielsen-Thurston classification states that any $f \in \Mod(S)$ is either periodic, reducible or pseudo-Anosov.\\

Pseudo-Anosovs act on $\mathcal{T}(S)$ in a similar way to a hyperbolic isometry: for any $f$ pseudo-Anosov there exists a bi-infinite geodesic in $\mathcal{T}(S)$ such that $f$ fixes this geodesic setwise and acts on it by translation. Denote the length of this translation by $l_{\mathcal{T}}(f).$ Using the definition of the Teichm\"uller metric one can see that $l_{\mathcal{T}}(f) = \log \lambda(f)$ where $\lambda(f)$ is the stretch factor of the pseudo-Anosov. Hence, measuring the translation length on Teichm\"uller space is equivalent to determining the stretch factor of $f.$ A big open question concerning the study of mapping classes is which numbers arise as a stretch factor of a pseudo-Anosov. It is known that for fixed genus $g,$ the set of all $\lambda$ which are stretch factors for some pseudo-Anosov of the genus $g$ surface is a closed, discrete subset of $\mathbb{R}$. In particular, this set has a minimum and it is an open problem to find the precise value of this minimum or the pseudo-Anosov attaining it. It is known however, that the minimum is of order $\frac{1}{g}.$ All of the above can be found in more detail in (\cite{primer}, Chapter 14).\\

The curve graph $\mathcal{C}(S)$ is a Gromov-hyperbolic graph and the action of pseudo-Anosovs on the curve graph is by hyperbolic isometries (\cite{MM}, Proposition 4.6). Here, a hyperbolic isometry is one for which the limit set of any orbit consists of exactly two points in the Gromov-boundary (as defined in \cite{G}, page 209). In fact, Bowditch showed that similarly to the picture in Teichm\"uller space, for a pseudo-Anosov $f$ there exists a bi-infinite geodesic in $\mathcal{C}(S)$ such that some power $f^k$ preserves this geodesic and acts on it by translation (\cite{Bowditch}). As before, we can measure the length of this translation, call it $l_{\mathcal{C}}(f^k)$ and set $l_C(f) := \frac{l_{\mathcal{C}}(f^k)}{k}$ to be the stable translation length of $f$ on the curve graph. In Section \ref{prelim} below, we give an equivalent definition of $l_{\mathcal{C}}(f)$ which justifies the name \textit{stable} curve graph translation length. From Bowditch's work, it follows that the $l_\mathcal{C}(f)$ are always rational.  As for stretch factors, the set of all numbers arising as stable curve graph translation lengths of pseudo-Anosovs of a fixed genus $g$ surface has a minimum. Again, it is an open problem to find the specific value of this minimum and the pseudo-Anosov attaining it. It is known that the minimum is of order $\frac{1}{g^2}$ (\cite{Vaibhav}).\\

Regarding the problem of finding the exact minima, one can relax the question by restricting to a subset of all pseudo-Anosovs and try to find the minimal stretch factor or stable curve graph translation length over this subset. Boissy--Lanneau restrict themselves to the set of pseudo-Anosovs that are affine with respect to a translation surface in a hyperelliptic stratum component. They proceed to determine the minimal stretch factor over this set and explicitly construct a pseudo-Anosov that attains this minimum (\cite{BL}). Let $f_g$ be this pseudo-Anosov for the respective genus $g$ surface. In particular, Boissy--Lanneau show that $\lambda(f_g) \ge \sqrt{2}$ for any $g.$\\

The first goal of this paper is to compute the stable curve graph translation length of the $f_g.$ We show:

\begin{theorem}\label{mytheorem}
    For all $g \ge 2,$ it holds that
    $$\frac{1}{16g-12} \, \le \, l_C(f_g) \, \le \, \frac{1}{g-1}.$$
\end{theorem}

We proceed with showing that the minimal stable curve graph translation length over all pseudo-Anosovs in a hyperelliptic component behaves like $l_\mathcal{C}(f_g).$ In the following, we write $f(n) \asymp h(n)$ for two functions $f, h: \mathbb{N}_{\ge 2} \to \mathbb{R}$ if $\frac{f(n)}{h(n)} \in [\frac{1}{C}, C]$ for some $C > 1$ and all $n.$

\begin{theorem}\label{bettertheorem}
    Let $l_g$ be the minimal stable curve graph translation length over all pseudo-Anosov that are affine with respect to a genus $g$ translation surface in a hyperelliptic component. Then:
    $$l_g \asymp \frac{1}{g}.$$
\end{theorem}

We remark that while the stable curve graph translation length of $f_g$ is of the same order as $l_g,$ it remains unclear whether $f_g$ is the actual minimiser of the stable curve graph translation length over all pseudo-Anosovs in a hyperelliptic component, i.e. whether $l_g = l_\mathcal{C}(f_g).$\\ 

Boissy--Lanneau's result can be interpreted as saying that pseudo-Anosovs in hyperelliptic components have large stretch factor. Theorem \ref{bettertheorem} says that also the stable curve graph translation lengths of such pseudo-Anosovs is large. The minimal stable curve graph translation length over all pseudo-Anosovs is of order $\frac{1}{g^2}$ and can hence certainly not be attained by a pseudo-Anosov in a hyperelliptic component. However, Theorem \ref{bettertheorem} also shows that after restricting to the pseudo-Anosovs in hyperelliptic components, the minimal stable curve graph translation length behaves differently to the minimal stretch factor. While the minimal stretch factor stays bounded from below by the constant $\sqrt{2},$ the minimal stable curve graph translation length behaves like $\frac{1}{g}$ and in particular tends to $0$ as we increase the genus.\\

This raises the question of how the two translation lengths $l_{\mathcal{T}}(f)$ and $l_{\mathcal{C}}(f)$ are related in general. Using the systole map from $\mathcal{T}$ to $\mathcal{C}$ which is coarsely Lipschitz, one can show that $l_{\mathcal{C}}(f) \le K \cdot l_{\mathcal{T}}(f)$ for any $f,$ where the constant $K$ depends only on the genus of $S$ and can be chosen to be approximately $\frac{1}{\log(g)}$ (see \cite{Vaibhav2} for details). In particular, this shows that small Teichm\"uller space translation length (or stretch factor) implies small stable curve graph translation length, but raises the question about the opposite. The $f_g$ from Theorem \ref{mytheorem} are an example of a sequence of pseudo-Anosovs such that $l_{\mathcal{T}}(f_g)$ is uniformly bounded away from $0,$ but $l_C(f_g) \longrightarrow 0$ as  $g \to \infty.$ One can ask if an even better statement is possible, i.e. if there exists a sequence $(h_g)$ with $h_g$ a pseudo-Anosov of a genus $g$ surface such that $l_{\mathcal{T}}(h_g) \longrightarrow \infty \text{ and } l_\mathcal{C}(h_g) \longrightarrow 0,$ as $g \to \infty.$ Another question would be what happens when we fix the genus, i.e. can we find a sequence of pseudo-Anosovs $(f_n) \subset \Mod(S)$ with $l_\mathcal{T}(f_n) \longrightarrow \infty$ as $n \to \infty$ while $l_\mathcal{C}(f_n) \le C$ for some constant $C > 0$ and all $n.$ Note that after fixing the genus, this is the best we can ask for since there is a minimal stable curve graph translation length, so we can't have $l_\mathcal{C}(f_n)$ approaching $0.$\\ 

We provide affirmative answers to the above questions by showing:

\begin{theorem}\label{myothertheorem}
    For any $g \ge 3,$ there exists a sequence $(f_n) \subset \Mod(S)$ of pseudo-Anosovs with 
    $$ \lim\limits_{n \to \infty} l_\mathcal{T}(f_n) = \infty \text{ and } l_\mathcal{C}(f_n) \le \frac{1}{g-1} \text{ for all } n \in \mathbb{N}.$$
\end{theorem}

We remark that in Theorem \ref{myothertheorem} the constant bounding the curve graph translation length can even be chosen to be of order $\frac{1}{g^2}.$ This will be discussed in the proof of the Theorem. However, it is enough to state the Theorem as above in order to obtain the following:

\begin{corollary}\label{cor1}
    There exists a sequence $(h_g)_{g=2}^\infty$, where $h_g$ is a pseudo-Anosov of a genus $g$ surface with
    $$l_{\mathcal{T}}(h_g) \to \infty \text{ and } l_\mathcal{C}(h_g) \to 0,$$ as $g \to \infty.$
\end{corollary}

Theorem \ref{myothertheorem} gives another interesting consequence about stable curve graph translation lengths, namely that there are such lengths with infinite multiplicity:

\begin{corollary}\label{cor2}
    For any $g \ge 3,$ there exists $q \in \mathbb{Q}$ such that there are infinitely many non-conjugate pseudo-Anosovs in $\Mod(S)$ with stable curve graph translation length $q.$
\end{corollary}

This is different from Teichm\"uller space translation length, where it is known that some $x \in \mathbb{R}$ is attained as $l_\mathcal{T}(f)$ by at most finitely many non-conjugate pseudo-Anosovs $f.$ We remark that the result of Corollary \ref{cor2} is folklore among experts in the field. For example, the pseudo-Anosovs constructed by Watanabe can also be used to obtain the Corollary (\cite{Wa}).\\

The above shows the existence of a number $q \in \mathbb{Q}$ with infinite multiplicity with respect to stable curve graph translation lengths. However, given a pseudo-Anosov $f,$ it is not clear whether $l_\mathcal{C}(f)$ is attained by an infinite number of non-conjugate pseudo-Anosovs. We end this work by presenting a method that for certain pseudo-Anosovs $f$ constructs infinitely many non-conjugate pseudo-Anosovs $f_i$ with $l_\mathcal{C}(f_i) = l_\mathcal{C}(f).$ In particular, we prove:

\begin{theorem}\label{construction of inf mult}
    Let $f$ be a pseudo-Anosov with $l_\mathcal{C}(f) = m,$ where $m \in \mathbb{N},$ and assume that there is a non-separating curve $\alpha$ attaining the translation length, i.e. 
    $$d_\mathcal{C}(\alpha, f^{k}(\alpha)) = km$$ for all $k \in \mathbb{Z}.$
    Then, for any pseudo-Anosov $h$ of the surface $S \setminus \alpha,$ there exists $N_h \in \mathbb{N}$ such that for $n \ge N_h$ the $f_{ h,n} := h^n \circ f$ satisfy $l_\mathcal{C}(f_{h,n}) = l_\mathcal{C}(f).$
\end{theorem}

The property of non-conjugacy of the $f_{h,n}$ is discussed in Section \ref{end}. By the surface $S \setminus \alpha$ in the above theorem, we mean the surface of genus $g-1$ with two boundary components resulting from removing an open, annular neighbourhood of $\alpha$ from $S.$ Note that the pseudo-Anosov $h$ fixes the boundary of $S \setminus \alpha$ pointwise and can hence be extended to a homeomorphism of $S,$ i.e. when we write $h \circ f$ we think of $h$ as an element of $\Mod(S).$ Of course $h$ seen as an element of $\Mod(S)$ is not a pseudo-Anosov but is reducible, since it fixes the curve $\alpha.$\\

As described above, every pseudo-Anosov admits a power with an invariant geodesic axis in the curve graph. Hence, if the axis consists of non-separating curves, then this power falls into the set of pseudo-Anosovs Theorem \ref{construction of inf mult} applies to. Another example of a pseudo-Anosov that satisfies the prerequisites of the Theorem is one constructed as a product of large powers of Dehn twists about two filling curves. Such a pseudo-Anosov has an invariant geodesic axis (\cite{BK}). Again, if this axis consists of non-separating curves, then Theorem \ref{construction of inf mult} applies.\\

It is not clear whether the conditions in Theorem \ref{construction of inf mult} can be relaxed or whether a similar construction can be used for all pseudo-Anosovs and the question whether all stable curve graph translation lengths are of infinite multiplicity remains open.\\

\textbf{Outline.} 
In Section \ref{prelim}, we introduce the basic notions needed for this work. In particular, we state the definitions of pseudo-Anosovs and their stable curve graph translation length and present some concepts of the theory of train tracks which will be useful for the proof of the lower bound in Theorem \ref{mytheorem}. In Section \ref{top RV}, we present Rauzy-Veech induction which is used in order to define the maps $f_g.$ However, since we are interested in $l_{\mathcal{C}}(f_g)$ which is a topological invariant, we omit the discussion of flat structures -which is normally used in order to define mapping classes through Rauzy-Veech induction- and present a purely topological version of it. In Section \ref{properties section}, we discuss some properties of the notions defined in Section \ref{top RV} that will be useful for the proofs. In Section \ref{proof section}, we define the pseudo-Anosovs $f_g$ and prove Theorems \ref{mytheorem} and \ref{bettertheorem}.\\
In Section \ref{pennerexample}, we prove Theorem \ref{myothertheorem} and the Corollaries \ref{cor1} and \ref{cor2}. Finally, in Section \ref{end}, we present the proof of Theorem \ref{construction of inf mult}. The last two sections use completely different methods from the rest of the sections and can be read independently. In particular, in order to prove Theorem \ref{myothertheorem}, we modify Penner's famous example of a pseudo-Anosov. For the proof of Theorem \ref{construction of inf mult}, we use a result about subsurface projections from Masur--Minsky.\\

\textbf{Acknowledgements.}
I would like to thank Vaibhav Gadre for valuable conversations and useful ideas regarding this work. Furthermore, I thank the anonymous referee for helpful comments and suggestions that improved the exposition.\\

\textbf{Funding.}
This work was supported by the Additional Funding Programme for Mathematical Sciences, delivered by EPSRC (EP/V521917/1) and the Heilbronn Institute for Mathematical Research.

\section{Preliminaries}\label{prelim}

\subsection{Pseudo-Anosovs and translation lengths}

Let $S$ be a closed genus $g$ surface. If not specified otherwise, we always assume that $g \ge 2.$ The mapping class group of $S$ is the group of orientation preserving self-homeomorphisms of $S$ up to isotopy. We denote this group by $\Mod(S).$

\begin{definition}
A homeomorphism $\phi$ of $S$ is called \textit{pseudo-Anosov} if there are two measured foliations $(\mathcal{F}^s, \mu_s), (\mathcal{F}^u, \mu_u)$ and $\lambda > 1$ such that
\begin{itemize}
    \item $F^s$ and $F^u$ are transverse
    \item $\phi \cdot (\mathcal{F}^s, \mu_s) = (\mathcal{F}^s, \lambda^{-1}\mu_s)$
    \item $\phi \cdot (\mathcal{F}^u, \mu_u) = (\mathcal{F}^u, \lambda\mu_u)$
\end{itemize}
\end{definition}

For a pseudo-Anosov $\phi,$ we call $\mathcal{F}^s, \mathcal{F}^u$ the stable and unstable foliation respectively and $\lambda$ the stretch factor (or the dilatation) of $\phi.$ A mapping class $f \in \Mod(S)$ is called pseudo-Anosov, if there exists a representative homeomorphism $\phi$ of $f$ that is pseudo-Anosov. For a given pseudo-Anosov $f \in \Mod(S)$ there are different pseudo-Anosov homeomorphisms representing $f,$ however they are all conjugate to each other. Hence, the stretch factor of $f$ is well-defined. We sometimes write $\lambda(f)$ instead of just $\lambda$ to emphasize the pseudo-Anosov the stretch factor corresponds to. See (\cite{primer}, Part 3) for more information on pseudo-Anosov theory.\\

For a general metric space $(X, d)$ and an isometry $f$ of $X,$ the \textit{translation length} of $f$ is defined as $$\inf\limits_{x \in X} d(x, f(x)).$$ The \textit{stable translation length} is given by $$\liminf\limits_{n \to \infty} \frac{d(x, f^n(x))}{n},$$
where $x \in X$ is any point.\\

\begin{lemma}
    The stable translation length is well defined.
\end{lemma}
\begin{proof}
    Let $y \neq x$ be any other point in $X.$ Then by the triangle inequality
    $$d(x, f^n(x)) \le d(x, y) + d(y, f^n(y)) + d(f^n(y), f^n(x))$$ 
    for any $n \in \mathbb{N}.$ Now since $f$ is an isometry of $X,$ we have $d(x, y) = d(f^n(x), f^n(y))$ and therefore
    $$d(x, f^n(x)) \le 2d(x, y) + d(y, f^n(y)).$$
    Dividing by $n$ and taking the lim inf on both sides yields
    $$\underset{n \to \infty}{\liminf} \ \frac{d(x, f^n(x))}{n} \le \underset{n \to \infty}{\liminf} \ \frac{d(y, f^n(y))}{n}.$$
    By reversing the roles of $x$ and $y$ one obtains the opposite inequality, which in total shows that the definition of the stable translation length is independent of the choice of the point $x$. 
\end{proof}

Let $\mathcal{T}(S)$ denote the Teichm\"uller space of $S.$ The mapping class group acts by isometries on $\mathcal{T}(S)$ with respect to the Teichm\"uller metric. For a pseudo-Anosov $f,$ its Teichm\"uller space translation length and stable translation length are the same and given by $\log(\lambda(f))$ (see \cite{primer}, Chapter 14 for details). We denote the (stable) Teichm\"uller space translation length of a pseudo-Anosov $f$ by $l_\mathcal{T}(f).$\\

Let $\mathcal{C}(S)$ denote the curve graph of $S.$ This is the graph with vertices corresponding to isotopy classes of essential simple closed curves, where two vertices are joined by an edge if there exist representative curves that are disjoint. Since homeomorphisms preserve the property of being an essential simple closed curve as well as disjointness of two such curves, the mapping class group $\Mod(S)$ acts on the curve graph $\mathcal{C}(S)$ by graph automorphisms. We equip $\mathcal{C}(S)$ with the path metric $d_\mathcal{C}$ where each edge has length $1.$ The mapping class group action is an action by isometries with respect to $d_\mathcal{C}.$ A pseudo-Anosov $f$ doesn't fix any finite set of curves on the surface and therefore its curve graph translation length is always greater or equal to $1.$ It is more interesting to consider the stable translation length.\\

For $f \in \Mod(S),$ the \textit{stable curve graph translation length} of $f$ is given by
$$l_\mathcal{C}(f) = \underset{n \to \infty}{\liminf} \ \frac{d_\mathcal{C}(\alpha, f^n(\alpha))}{n},$$
where $\alpha \in \mathcal{C}(S)$ is any vertex. From now on, when it comes to the curve graph, we will only consider the stable translation length.  

\subsection{Train tracks}\label{trains}

In this section, we introduce the notion of train tracks and discuss their importance for pseudo-Anosov maps. As before, let $S$ denote a closed surface of genus $g.$ For the following definition, we equip $S$ with a smooth structure.

\begin{definition}
    A \textit{train track} $\tau$ on $S$ is an embedded graph such that each edge is a smooth path and at each vertex all adjacent edges are mutually tangent. 
\end{definition}

Vertices of $\tau$ are usually referred to as switches and edges of $\tau$ as branches. The tangency condition yields a splitting of the set of branches adjacent to a switch into two sets, the incoming and outgoing branches. In the following, we mention some important notions that we make use of. For a more detailed introduction to train tracks see \cite{PH}.\\

We say that a train track is \textit{large} if all components of $S \setminus \tau$ are polygons.\\

A \textit{train route} is a smooth path in $\tau.$ Hence, a train route passing through a switch can only pass from an incoming to an outgoing branch or the other way around. A curve $\gamma$ on $S$ is \textit{carried} by $\tau$ if it is isotopic to a closed train route. A train track is \textit{recurrent} if every branch is contained in a closed train route. A train track is \textit{transversely recurrent} if for every branch there exists a simple closed curve intersecting the branch efficiently, i.e. there are no bigons between the branch and the curve. Finally, a train track is \textit{birecurrent} if it is both recurrent and transversely recurrent.\\

A \textit{measure} $\mu$ on $\tau$ is an assignment of a non-negative real number (also called a \textit{weight}) to each branch of $\tau$ such that at each switch the sum of weights of the incoming branches equals the sum of weights of the outgoing ones. Denote by $P(\tau)$ the set of measures on $\tau.$ Note that a curve carried by $\tau$ induces the counting measure on $\tau.$ Hence, we can think of such a curve as an element of $P(\tau).$ Let $\text{int}(P(\tau))$ denote the set of measures that are positive on each branch.\\

A diagonal extension of $\tau$ is a train track $\tau'$ which consists of the same switches and branches as $\tau$ and possibly has some extra branches (called \textit{diagonals}) which start and terminate in corners of some complementary region of $S \setminus \tau.$ Let $E(\tau)$ be the set of diagonal extensions of $\tau$ and $PE(\tau)$ the union of the sets of measures of all diagonal extensions. Furthermore, let $\text{int}(PE(\tau))$ be the set of measures in $PE(\tau)$ that are positive on each branch of $\tau.$\\ 

The next lemma, often referred to as the \textit{Nesting Lemma}, is crucial in our computation of the lower bound in Theorem \ref{mytheorem}. In the following, one should think of the sets $PE(\tau)$ and $\text{int}(PE(\tau))$ as subsets of the curve graph $\mathcal{C}(S)$, i.e. only consider the measures coming from curves, and $\mathcal{N}_1$ denotes the 1-neighbourhood of a subset of $\mathcal{C}(S).$

\begin{lemma}
    Let $\tau$ be a large recurrent train track. Then 
    $$\mathcal{N}_1(\text{int}(PE(\tau))) \subset PE(\tau).$$
\end{lemma}

\begin{proof}
    The Nesting Lemma is due to Masur--Minsky (\cite{MM}, Lemma 4.4). They prove it for the stronger assumption of the train track being birecurrent. The weakening of this to require the train track to only be recurrent is done in (\cite{Vaibhav}, Lemma 3.2). 
\end{proof}

A train track $\sigma$ is \textit{carried} by the train track $\tau$ if there is an isotopy of $S$ that takes every train route of $\sigma$ to a train route of $\tau.$ We write $\sigma \prec \tau.$ A train track is \textit{invariant} for a homeomorphism $f$ if $f(\tau) \prec \tau.$ In this case, we can define a matrix $V_f$ as follows: Let $n$ be the number of branches of $\tau.$ Let $b_1, ... , b_n$ be a numbering of the branches and $V_f$ be the $n \times n$ matrix whose $(i,j)$-entry counts how often the branch $f(b_j)$ passes over the branch $b_i.$ Hence, the $j^{\text{th}}$ column of $V_f$ describes the train route $f(b_j).$ We call $V_f$ a \textit{train track matrix} for $f.$ Note that $V_f$ depends on the invariant train track $\tau$ as well.\\

Train tracks play an important role in the theory of pseudo-Anosovs, because for every pseudo-Anosov, there exists an invariant train track. The corresponding train track matrix is a Perron-Frobenius matrix and its Perron-Frobenius eigenvalue is the stretch factor of the pseudo-Anosov (\cite{primer}, Chapter 15.2). We will use an invariant train track and its corresponding train track matrix to compute the lower bound in Theorem \ref{mytheorem}. 

\subsection{Translation surfaces and affine pseudo-Anosovs}\label{affine pA}

We recall the necessary definitions of the terms that come up in Theorem \ref{bettertheorem}.

\begin{definition}
    A \textit{translation surface} is a pair $(X, \omega)$ where $X$ is a Riemann surface and $\omega$ a non-zero holomorphic $1$-form on $X.$ 
\end{definition}  

The name translation surface is justified by the fact that a translation surface admits a natural atlas of charts away from the zeros of $\omega$ whose transition functions are translations (see, for example, \cite{AW}). The zeros of $\omega$ are usually referred to as \textit{singularities}.\\

A homeomorphism $f$ of a translation surface $(X, \omega)$ is called \textit{affine} if it sends singularities to singularities and is affine on the natural charts. Note that this is well-defined since the chart transitions are translations, so in particular affine maps.\\

A translation surface $(X, \omega)$ is hyperelliptic if the underlying Riemann surface $X$ is hyperelliptic.\\

Given a translation surface $(X, \omega),$ let $x_1, ... , x_n$ be the zeros of $\omega$ and let $k_1, ... , k_n$ be the orders of the $x_i$ respectively. The Riemann-Roch Theorem ensures that $k_1 + ... + k_n = 2g-2$ where $g$ is the genus of $X.$ Thus, for any given tuple $(k_1, ... , k_n)$ with $k_1+...+k_n = 2g-2,$ one can define the space $\mathcal{H}(k_1, ... , k_n)$ to be the space of all translation surfaces with exactly $n$ singularities of orders $k_1, ... , k_n.$ These spaces are referred to as \textit{strata} of translation surfaces. Kontsevich and Zorich classify the connected components of these strata (\cite{KZ}). In particular, the classification shows that the two strata $\mathcal{H}(2g-2)$ and $\mathcal{H}(g-1, g-1)$ contain a connected component that consists entirely of hyperelliptic translation surfaces. These connected components are referred to as \textit{hyperelliptic components}.\\

If $S$ is a topological closed genus $g$ surface, then we say that $f \in \Mod(S)$ is affine for a translation surface in a hyperelliptic component (or simply $f$ is in a hyperelliptic component) if there exists a representative $\phi$ of $f,$ a translation surface $(X, \omega)$ in a hyperelliptic component, a homeomorphism $\psi: S \to X$ and an affine homeomorphism $\phi_X$ of $(X, \omega)$ such that the following diagram commutes:

\begin{center}
    \begin{tikzcd}
        S \arrow{r}{\phi} \arrow{d}{\psi} & S \arrow{d}{\psi}\\
        X \arrow{r}{\phi_X} & X.
    \end{tikzcd}
\end{center}

\section{Topological Rauzy-Veech induction}\label{top RV}

In this section, we present the well-known method of using Rauzy-Veech induction to obtain mapping classes. This relies on the theory of translation surfaces. However, since we are interested in the stable curve graph translation length, which is a topological invariant of a pseudo-Anosov, we present a purely topological version of Rauzy-Veech induction that omits the discussion of translation surfaces.

\subsection{Rauzy-Veech induction on permutations}\label{RV induction}

Given an interval exchange map, Rauzy-Veech induction is a method to produce a new interval exchange map out of the given one. An interval exchange map together with some extra data gives rise to a translation surface. Therefore, the study of translation surfaces is closely related to that of interval exchange maps and Rauzy-Veech induction can be used to define homeomorphisms that are affine with respect to a certain translation surface. We refer the reader to \cite{Yoc} for an introduction to this theory. Boissy--Lanneau use Rauzy-Veech induction in order to define pseudo-Anosov maps that are affine with respect to some translation surface in the hyperelliptic stratum component of $\mathcal{H}(2g-2)$ and find the pseudo-Anosov with minimal stretch factor out of all of these (\cite{BL}). For given $g,$ we denote the pseudo-Anosov that attains this minimum by $f_g.$ Boissy--Lanneau proceed to do the same with the hyperelliptic stratum component of $\mathcal{H}(g-1, g-1).$ However, the stretch factor of $f_g$ is less than the minimal one for $\mathcal{H}(g-1, g-1),$ and hence it is the minimal one over all hyperelliptic components. This is why we focus on $f_g.$ We point out though that the techniques presented in the computation of $l_\mathcal{C}(f_g)$ below can be used to compute the stable curve graph translation length of any pseudo-Anosov defined in terms of Rauzy-Veech induction.\\

In this work, our goal is to compute the stable curve graph translation length of the $f_g.$ In order to define the pseudo-Anosovs $f_g$ (as in \cite{BL}, Section 5), we use Rauzy-Veech induction. Since we are interested in $l_\mathcal{C}(f_g)$ which is a purely topological invariant, we omit the discussion of translation surfaces and present a more topological version of Rauzy-Veech induction. In particular, instead of an interval exchange map, we require only a permutation.\\ 

Let $\pi \in S_n$ be a permutation. We say that $\pi$ is \textit{irreducible}, if $\pi(\{1, ... , k\}) = \{1, ... , k\}$ implies $k=n.$ Throughout, we assume that all permutations are irreducible. Note that this ensures that $\pi(n) \neq n$ which is necessary in order to define the Rauzy-Veech induction below.\\

We want to work with so called labeled permutations. Let $\mathcal{A}$ be a finite alphabet with $n$ elements and $\pi_t: \mathcal{A} \to \{1, ... ,n\}$ and $\pi_b: \mathcal{A} \to \{1, ... ,n\}$ bijections, such that $\pi = \pi_b \circ \pi^{-1}_t.$ We call the pair $(\pi_t, \pi_b)$ a \textit{labeled permutation} and say that this labeled permutation is irreducible, if the corresponding $\pi$ is irreducible. Here, the subscripts $t$ and $b$ stand for "top" and "bottom" respectively. Note that $\pi$ does not determine $\pi_t$ and $\pi_b$ uniquely. In other words, after fixing an alphabet $\mathcal{A},$ we have a surjective, but not injective map from the set of labeled permutation to $S_n.$\\

As an example, consider $\pi = (1 \ 2 \ 3) \in S_3$ and $\mathcal{A} = \{A, B, C\}.$ One choice for $\pi_t, \, \pi_b$ would be $\pi_t(A) = 1, \, \pi_t(B) = 3, \, \pi_t(C) = 2$ and $\pi_b(A) = 2, \, \pi_b(B) = 1, \, \pi_b(C) = 3.$ We usually represent our labeled permutation as a matrix

$$\begin{pmatrix}
    \pi^{-1}_t(1) & ... & \pi^{-1}_t(n)\\
    \pi^{-1}_b(1) & ... & \pi^{-1}_b(n)\\
\end{pmatrix},$$

which in our example gives

$$\begin{pmatrix}
    A & C & B\\
    B & A & C\\
\end{pmatrix}.$$

Given a labeled permutation $(\pi_t, \pi_b),$ we construct a surface $X(\pi_t, \pi_b)$ by taking a regular $2n$-gon, labeling its top sides according to $\pi_t$ and its bottom sides according to $\pi_b$ and gluing the sides with the same label in the way they are oriented in Figure \ref{fig:2n-gon}. Here, the gluings are just topological and not necessarily by translations. In order to distinguish between the polygon and the resulting surface after side identifications, we denote the polygon by $P(\pi_t, \pi_b).$ Figure \ref{fig:2n-gon} shows an example for the labeled permutation $$\begin{pmatrix}
    A & B & C & D\\
    D & A & C & B\\
\end{pmatrix}.$$

\begin{figure}[h]
    \centering
    \begin{tikzpicture}
        \node[anchor=south west,inner sep=0] at (0,0){\includegraphics[scale=0.5]{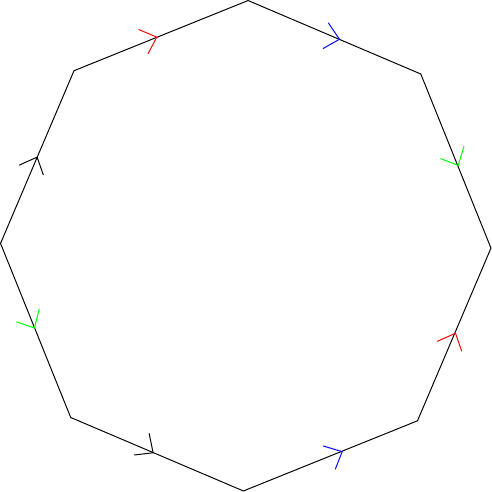}};
        \node at (0,3) {$A$};
        \node at (1.3, 4.2) {$\textcolor{red}{B}$};
        \node at (3, 4.2) {$\textcolor{blue}{C}$};
        \node at (4.15,3) {$\textcolor{green}{D}$};
        \node at (1,0) {$A$};
        \node at (4.1, 1.1) {$\textcolor{red}{B}$};
        \node at (3, 0) {$\textcolor{blue}{C}$};
        \node at (0,1.1) {$\textcolor{green}{D}$};
    \end{tikzpicture}
    \caption{An example of the $2n$-gon construction for $n = 4$} 
    \label{fig:2n-gon}
\end{figure}

We now define Rauzy-Veech induction. Let $(\pi_t, \pi_b)$ be a labeled permutation with $\pi = \pi_b \circ \pi^{-1}_t.$ Rauzy-Veech induction consists of two purely combinatorial ways to obtain a new labeled permutation from the given one. The two ways are called top, resp. bottom Rauzy-Veech induction, which we denote by $\mathcal{R}_t$ and $\mathcal{R}_b$ respectively.

\begin{itemize}
    \item $\mathcal{R}_t$:\\ 
    Let $\alpha := \pi^{-1}_t(n)$ and $\beta := \pi^{-1}_b(n).$ In this case, we say that the top is the winner and the bottom the loser, or equivalently $\alpha$ is the winner, whereas $\beta$ is the loser. We refer to the Rauzy-Veech induction in this case as being of type $t.$\\
    We define the Rauzy-Veech induction $\mathcal{R}_t(\pi_t, \pi_b)$ as $(\pi'_t, \pi'_b),$ where $\pi'_t = \pi_t$ and $\pi'_b$ differs from $\pi_b$ by moving $\beta$ to the right of $\alpha$ and translating everything afterwards to the right by one. For example
    $$\begin{pmatrix}
    A & B & C & D\\
    D & C & B & A\\
\end{pmatrix} \xlongrightarrow{\mathcal{R}_t} \begin{pmatrix}
    A & B & C & D\\
    D & A & C & B\\
\end{pmatrix}.$$

    \item $\mathcal{R}_b$:\\ 
    Let $\alpha := \pi^{-1}_t(n)$ and $\beta := \pi^{-1}_b(n).$ In this case, we say that the bottom is the winner and the top the loser, or equivalently $\beta$ is the winner, whereas $\alpha$ is the loser. We refer to the Rauzy-Veech induction in this case as being of type $b.$\\
    We define the Rauzy-Veech induction $\mathcal{R}_b(\pi_t, \pi_b)$ as $(\pi'_t, \pi'_b),$ where $\pi'_b = \pi_b$ and $\pi'_t$ differs from $\pi_t$ by moving $\alpha$ to the right of $\beta$ and translating everything afterwards to the right by one. For example
    $$\begin{pmatrix}
    A & B & C & D\\
    D & C & B & A\\
\end{pmatrix} \xlongrightarrow{\mathcal{R}_b} \begin{pmatrix}
    A & D & B & C \\
    D & C & B & A\\
\end{pmatrix}.$$
\end{itemize}

We define a graph called the \textit{labeled Rauzy diagram} $\mathcal{D}_n$. The vertices of this graph correspond to irreducible labeled permutations of fixed length $n$ and we connect two vertices $(\pi_t, \pi_b), (\pi'_t, \pi'_b)$ by a with $t$ resp. $b$ labeled, oriented edge (from $(\pi_t, \pi_b)$ to $(\pi'_t, \pi'_b)$) if $\mathcal{R}_t(\pi_t, \pi_b) = (\pi'_t, \pi'_b)$ resp. $\mathcal{R}_b(\pi_t, \pi_b) = (\pi'_t, \pi'_b).$\\

The labeled Rauzy diagram is in general not a connected graph, but we always restrict ourselves to a connected component. We will be interested in the connected component of the labeled permutation $$\begin{pmatrix}
    \alpha_1 & \alpha_2 & ... & \alpha_n\\
    \alpha_n & \alpha_{n-1} & ... & \alpha_1\\
\end{pmatrix}.$$
For example, if $n = 3,$ the connected component of $$\begin{pmatrix}
    A & B & C \\
    C & B & A\\
\end{pmatrix}$$ 
is

$$ \text{\tiny t}\circlearrowright \begin{pmatrix}
    A & C & B \\
    C & B & A\\
\end{pmatrix} \substack{\xlongrightarrow{b}\\ \xlongleftarrow{b}} \begin{pmatrix}
    A & B & C \\
    C & B & A\\
\end{pmatrix} \substack{\xlongrightarrow{t}\\ \xlongleftarrow{t}} \begin{pmatrix}
    A & B & C \\
    C & A & B\\ 
\end{pmatrix} \circlearrowleft \text{\tiny{b}}.\\[2ex]$$

Finally, we want to assign a matrix to every edge of the labeled Rauzy diagram. Let $e$ be an edge of $\mathcal{D}_n.$ Let $(\pi_t, \pi_b)$ be the initial vertex of $e.$ Since $e$ corresponds to a Rauzy-Veech move (either $t$ or $b$), it determines a winner-loser pair $(\alpha, \beta).$ Let $E_{\alpha, \beta}$ be the matrix with a single non-zero entry equal to $1$ in the $(\alpha, \beta)$ position and let $V_{\alpha, \beta} = Id + E_{\alpha, \beta}.$ We assign to $e$ the matrix $V_{\alpha, \beta}.$

\subsection{Construction of mapping classes through Rauzy-Veech induction}\label{construction}

Let $(\pi_t, \pi_b)$ be a labeled permutation, $P(\pi_t, \pi_b)$ the corresponding polygon and $X(\pi_t, \pi_b)$ the corresponding surface. We want to interpret Rauzy-Veech induction on $(\pi_t, \pi_b)$ in a geometric way, i.e. relate $X(\pi_t, \pi_b)$ and $X(\mathcal{R}(\pi_t, \pi_b))$. Consider the following cut and paste process on $P(\pi_t, \pi_b).$\\

Let $T$ be the triangle consisting of the top side corresponding to $\alpha := \pi_t^{-1}(d),$ the bottom side corresponding to $\beta := \pi_b^{-1}(d)$ and the straight line joining their two distinct endpoints, see Figure \ref{fig:triangle}. Cut the triangle $T$ and glue it back by either gluing the top side $\alpha$ to the bottom side corresponding to $\alpha$ or gluing the bottom side $\beta$ to the top side corresponding to $\beta.$ After having glued back $T,$ we choose an ambient isotopy of the plane to isotope the resulting polygon back to being a regular $2n$-gon, see Figure \ref{fig:cut and paste}.\\ 

\begin{figure}[h]
    \centering
    \begin{tikzpicture}
    
        \node[anchor=south west,inner sep=0] at (0,0){\includegraphics[scale=0.7]{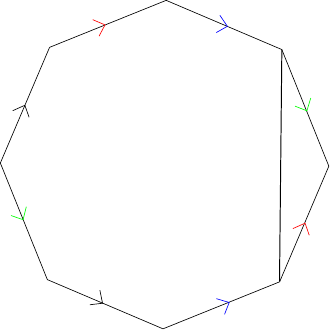}};
        \node at (3.6,1.95) {$T$};
        \node at (4, 2.8) {$\textcolor{green}{\alpha}$};
        \node at (4, 1.2) {$\textcolor{red}{\beta}$};
    \end{tikzpicture}
    \caption{The triangle $T$}
    \label{fig:triangle}
\end{figure}

\begin{figure}[h]
    \centering
    \includegraphics[scale=0.7]{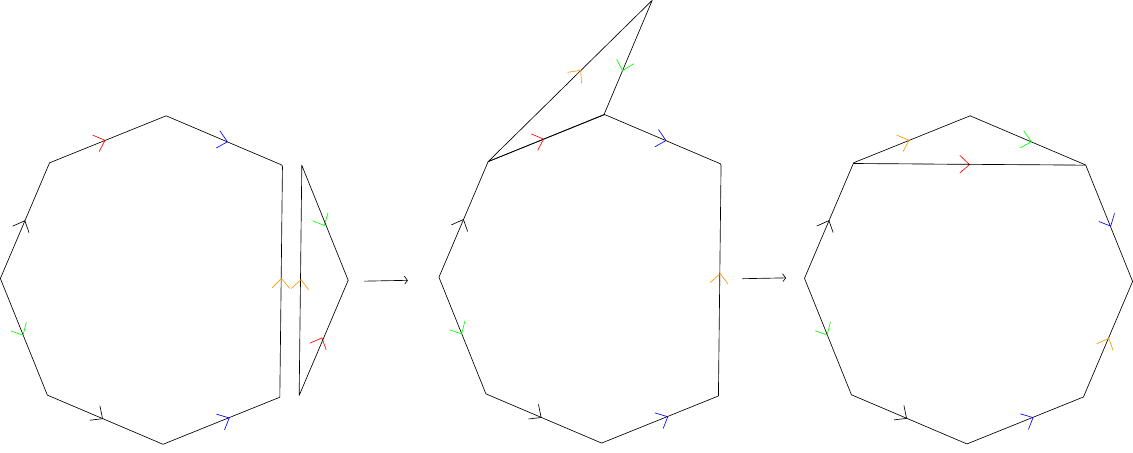}
    \caption{The cut and paste construction for a bottom move}
    \label{fig:cut and paste}
\end{figure}

We refer to the case where we glue the $\alpha$'s as a top move and to the case where we glue the $\beta$'s as a bottom move. In both cases, the new polygon has two new unlabeled sides, which we label by $\alpha$ in the case of a top move and by $\beta$ in the case of a bottom move.\\

 In total, we obtain that performing a top or bottom move to $P(\pi_t, \pi_b)$ as described above yields a new $2n$-gon together with a labeled permutation that describes the gluing of the sides, where this new labeled permutation is exactly obtained by Rauzy-Veech induction of type $t$ resp. $b$ on $(\pi_t, \pi_b).$ After gluing sides according to the new permutation, we obtain the surface $X(\mathcal{R}(\pi_t, \pi_b)).$\\

In summary, we now have a way to assign to a given labeled permutation a surface and interpret a Rauzy-Veech move (equivalently an outgoing edge in the labeled Rauzy diagram) as a cut and paste move of a triangle in order to obtain the surface corresponding to the resulting labeled permutation. We now want to use this interpretation to define mapping classes.\\

Consider an edge path $\gamma$ in the labeled Rauzy diagram, so that its endpoints define the same unlabeled permutation. We call such a path an \textit{allowed path}. In other words, an allowed path $\gamma$ is a finite sequence of $t$'s and $b$'s such that its starting point $(\pi_t, \pi_b)$ and endpoint $(\pi'_t, \pi'_b)$ satisfy $\pi_b \circ \pi^{-1}_t = \pi'_b \circ \pi'^{-1}_t.$ Since the unlabeled permutations are the same, the associated surfaces $X(\pi_t, \pi_b)$ and $X(\pi'_t, \pi'_b)$ admit a natural homeomorphism $\phi: X(\pi_t, \pi_b) \to X(\pi'_t, \pi'_b),$ which just comes by identifying the $2n$-gons $P(\pi_t, \pi_b)$ and $P(\pi'_t, \pi'_b)$ via the identity map. We refer to $\phi$ as the change of labeling. There is a second homeomorphism from $X(\pi_t, \pi_b)$ to $X(\pi'_t, \pi'_b)$ which is given by keeping track of the moves of $\gamma,$ i.e. sending each point to the corresponding point after each cut and pasting step. We denote this homeomorphism by $\psi.$ Note that at each cut and paste step we have chosen an ambient isotopy of the plane to obtain a regular $2n$-gon (cf. right arrow in Figure \ref{fig:cut and paste}). Any other choice of isotopies would result in an isotopic homeomorphism $\psi'.$ For us, it is enough to define $\psi$ up to isotopy, since we only need a well defined mapping class. We refer to $\psi$ as the cut and paste homeomorphism.\\

For an allowed path $\gamma,$ we let $f_\gamma := \psi^{-1} \circ \phi  \in \Mod(X(\pi_t, \pi_b)),$ where $(\pi_t, \pi_b)$ is the starting point of $\gamma$ and $\phi, \psi$ are as above.\\

To $f_\gamma$ we associate a matrix as follows: Let $e_1, ... , e_k$ be the edges of $\gamma$ ordered from first to last, i.e. the starting point of $e_1$ is $(\pi_t, \pi_b)$ and the endpoint of $e_k$ is $(\pi'_t, \pi'_b).$ Let $V_{e_i}$ be the corresponding matrices as discussed at the end of Section \ref{RV induction}. Finally, let $P$ be the permutation matrix that is $1$ in the $(\alpha, \beta)$ entry if and only if $\phi^{-1}(\alpha) = \beta,$ or equivalently $\phi(\beta) = \alpha,$ and $0$ otherwise. Here, $\phi(\beta) = \alpha$ is to be understood as $\phi$ maps the side labeled $\beta$ in the $2n$-gon to the side labeled $\alpha$ in the other $2n$-gon. We let $V_\gamma := V_{e_1} \cdot ... \cdot V_{e_k} \cdot P.$\\

Note that $V_\gamma$ is a matrix with only non-negative entries. We say that $V_\gamma$ is \textit{primitive} if there exists a $k \in \mathbb{N}$ such that $V_\gamma^k$ is positive, i.e. consists of strictly positive entries. In \cite{Veech}, it is shown that whenever the matrix $V_\gamma$ is primitive, the corresponding mapping class $f_\gamma$ is a pseudo-Anosov. Furthermore, the spectral radius of $V_\gamma$ is the stretch factor of $f_\gamma.$

\section{Some properties of $f_\gamma$ and $V_\gamma$}\label{properties section}

Let $\gamma$ be an allowed path in the labeled Rauzy diagram. In this section, we want to understand how the matrix $V_\gamma$ describes the map $f_\gamma.$ Our first goal is to show that there is an invariant train track $\tau_\gamma$ for $f_\gamma$ such that $V_\gamma$ is the corresponding train track matrix.\\

Let $(\pi_t, \pi_b)$ be the starting point of $\gamma$ and let $P_\gamma := P(\pi_t, \pi_b), \ X_\gamma := X(\pi_t, \pi_b)$ be the corresponding polygon and surface. Let $\tau_\gamma$ be the train track on $X_\gamma$ which consists of a single switch in the middle of $P_\gamma$ and $n$ branches, each of which goes from the switch to a bottom side of $P_\gamma,$ comes back from the corresponding top side of $P_\gamma$ and connects back to the switch. For the labeled permutation 

$$\begin{pmatrix}
    A & B & C & D\\
    D & A & C & B\\
\end{pmatrix}$$

an example can be seen in Figure \ref{fig:train track}.

\begin{figure}[h]
    \centering
    \includegraphics[scale=0.35]{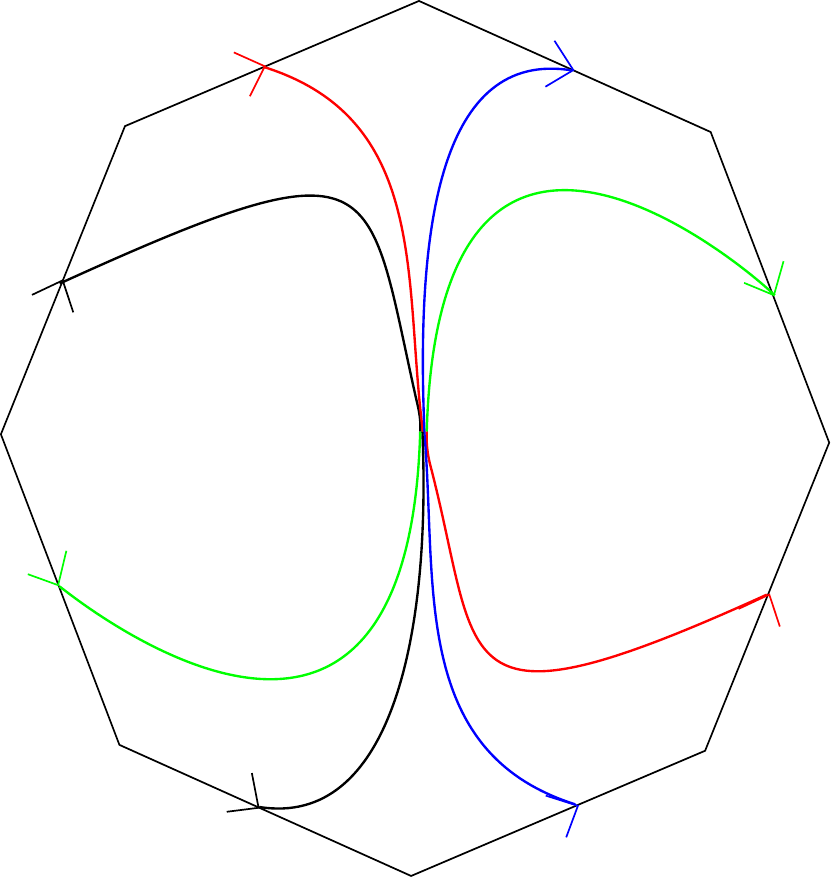}
    \caption{The train track $\tau$}
    \label{fig:train track}
\end{figure}

Label each branch of $\tau_\gamma$ with the same alphabet used to define the labeled permutations $\pi_t, \pi_b,$ i.e. call a branch $\alpha$ if it connects the sides of $P_\gamma$ labeled $\alpha.$ We show:

\begin{lemma}\label{train track lemma}
    $\tau_\gamma$ is an invariant train track for $f_\gamma$ and the corresponding train track matrix is $V_\gamma.$
\end{lemma}

\begin{proof}
    Let $(\pi_t', \pi_b')$ be the endpoint of the path $\gamma.$ Recall that with the notation of Section \ref{construction} $f_\gamma = \psi^{-1} \circ \phi,$ where $\phi, \psi: X_\gamma \to X(\pi_t', \pi_b')$ are the change of labeling and cut and paste homeomorphism respectively.\\

    Hence, in order to determine $f_\gamma(\tau_\gamma),$ we need to apply $\phi$ first, followed by doing all the cut and paste moves backwards. $\phi$ maps $\tau_\gamma$ to a train track on $X(\pi_t', \pi_b')$ which when drawn on $P(\pi_t', \pi_b')$ looks exactly the same as $\tau_\gamma.$ However, the labels of the branches may have changed. This is exactly captured by the permutation matrix $P$ defined in Section \ref{construction}. Any branch $\alpha$ of $\tau_\gamma$ is mapped under $\phi$ to the branch $\phi(\alpha) =: \beta$ and the $\alpha$-column of $P$ has exactly one non-zero entry equal to $1$ in the $(\beta, \alpha)$ position.\\

    For the cut and paste moves, we analyze a single $b$-move (a $t$-move is analogous) and then argue by induction. So, assume we apply the inverse of a $b$-move with winner-loser pair $(\alpha, \beta).$ This corresponds to cutting the triangle consisting of the top sides $\alpha$ and $\beta$ and a straight line connecting them and gluing the $\alpha$-side of it to the bottom side labeled $\alpha,$ (compare Figure \ref{fig:train track effect}, where $(\alpha, \beta) = (B,C).$ In the middle polygon of the figure, we labeled the "new" $B$-edge). The effect on our train track is now easy to observe. The $\alpha$-branch will still be the $\alpha$-branch in the new polygon, while the $\beta$-branch can be isotoped to run exactly once over itself and once over the $\alpha$-branch. All other branches will remain the same. This is exactly captured by the matrix $V_{\alpha, \beta}.$ For any $\alpha' \neq \beta,$ the $\alpha'$-column consists of a single non-zero entry equal to $1$ in the $(\alpha', \alpha')$ position. The $\beta$-column has two non-zero entries equal to $1$ in the positions $(\beta, \beta)$ and $(\alpha, \beta).$\\

    \begin{figure}[h]
    \centering
    \begin{tikzpicture}
    
        \node[anchor=south west,inner sep=0] at (0,0){\includegraphics[scale=0.7]{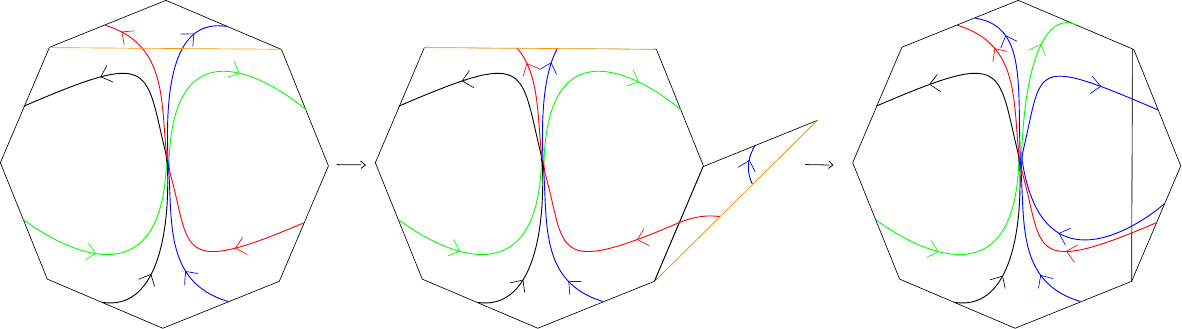}};
        \node at (0,2.8) {$A$};
        \node at (1.1, 3.9) {$B$};
        \node at (2.9, 3.9) {$C$};
        \node at (3.9,2.8) {$D$};
        \node at (0,1.2) {$D$};
        \node at (1.1, 0) {$A$};
        \node at (2.9, 0) {$C$};
        \node at (3.9,1.2) {$B$};

        \node at (6.4,3.6) {$\textcolor{orange}{B}$};
        \node at (8.9,1.4) {$\textcolor{orange}{B}$};

        \node at (10.1,2.8) {$A$};
        \node at (11.2, 3.9) {$B$};
        \node at (13, 3.9) {$D$};
        \node at (14,2.8) {$C$};
        \node at (10.1,1.2) {$D$};
        \node at (11.2, 0) {$A$};
        \node at (13, 0) {$C$};
        \node at (14,1.2) {$B$};
    \end{tikzpicture}
    
    \caption{The inverse of a $b$-move with its effect on the train track}
    \label{fig:train track effect}
\end{figure}

    Hence, we conclude inductively that after applying $\phi$ and each cut and paste move reversed, our newly obtained train track $f_\gamma(\tau_\gamma)$ is carried by $\tau_\gamma.$ Furthermore, if $e_1, ... , e_k$ are the edges of $\gamma,$ then the $\alpha$-column of the matrix $$V_{e_1} \cdot ... \cdot V_{e_k} \cdot P = V_\gamma$$ describes the image of the branch $f_\gamma(\alpha)$ for any $\alpha,$ i.e. $V_\gamma$ is the train track matrix.  
\end{proof}

The above lemma will be used in the proof of the lower bound in Theorem \ref{mytheorem}. For the proof of the upper bound, it will be more convenient to observe what happens to the sides of the polygon $P_\gamma.$ Note that these sides do not necessarily define closed curves in $X_\gamma.$ However, it will turn out that for the pseudo-Anosovs we will be interested in, the sides do define closed curves. In general, the sides only define paths on $X_\gamma,$ but a concatenation of these paths defines a closed curve and this in turn can be used to apply a similar method as the one we present in Section \ref{proof section} to compute upper bounds of the stable curve graph translation length in other examples.\\

Since $l_\mathcal{C}(f_\gamma) = l_\mathcal{C}(f_\gamma^{-1}),$ it does not matter if we use $f_\gamma$ or $f_\gamma^{-1}$ to compute the translation length. In fact, we want to study the action of $f_\gamma^{-1}$ on the sides of $P_\gamma$ instead of the one of $f_\gamma.$ We make this choice because $f_\gamma^{-1} = \phi^{-1} \circ \psi$ (compare the notation of Section \ref{construction}) and it is slightly easier to apply the cut and paste homeomorphism $\psi$ instead of applying its inverse. The reason we didn't look at $f_\gamma^{-1}$ to begin with is that the train track $\tau_\gamma$ is not carried by $f_\gamma^{-1}$ but only by $f_\gamma.$\\

Let $\alpha$ be a side of the polygon $P_\gamma.$ We think of $\alpha$ as a side of $P_\gamma$ as well as a path in $X_\gamma$ interchangeably as long as there is no ambiguity.

\begin{lemma}\label{lemma1}
    The path $f_\gamma^{-1}(\alpha)$ can be homotoped so that it lies only on the sides of the polygon $P_\gamma.$ Furthermore, the $(\alpha, \beta)$ entry of $V_\gamma$ counts exactly how often $f^{-1}_\gamma(\alpha)$ passes over $\beta.$ 
\end{lemma}

\begin{proof}
    Since $f_\gamma^{-1} = \phi^{-1} \circ \psi,$ we need to first apply the cut and paste moves to $P_\gamma$ and then the inverse of the change of labeling.\\
    
    Consider a single edge $e$ of $\gamma$ corresponding to a bottom cut and paste move with winner-loser pair $(\alpha, \beta).$ The case of a top move is analogous.\\

    Any side which does not correspond to winner side remains a side of the new polygon with the same label (compare Figure \ref{fig:cut and paste}). The side $\alpha$ gets mapped to a line connecting two corners of the new polygon, which can be homotoped to pass exactly once over the new side labeled $\alpha$ and the side labeled $\beta$ (compare Figure \ref{fig:homotopy of alpha}). This is captured by the rows of the matrix $V_{\alpha, \beta}.$ For any $\beta' \neq \alpha,$ the $\beta'$-row of $V_{\alpha, \beta}$ has a single non-zero entry equal to $1$ in the $(\beta', \beta')$ position, while the $\alpha$-row has two non-zero entries equal to $1$ in positions $(\alpha, \alpha)$ and $(\alpha, \beta).$\\

    \begin{figure}
        \centering
        \includegraphics[scale=0.7]{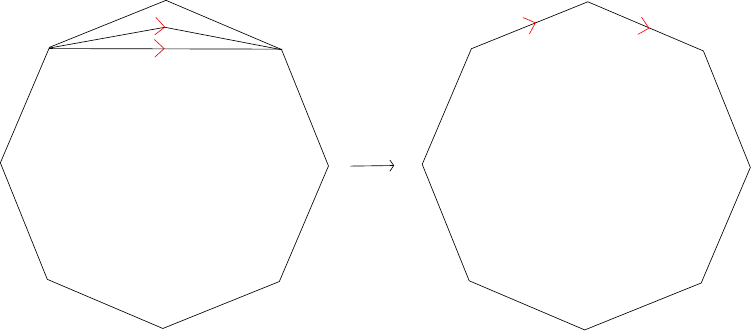}
        \caption{The homotopy of the side $\alpha$ after the cut and paste move}
        \label{fig:homotopy of alpha}
    \end{figure}

    Hence, inductively, we obtain that after applying all cut and paste moves, the image of the sides of $P_\gamma$ can be homotoped to lie on the sides of $P(\pi_t', \pi_b'),$ and letting $e_1, ... , e_k$ be the edges of $\gamma,$ the $\alpha$-row of the matrix $V_{e_1} \cdot ... \cdot V_{e_k}$ describes the image of $\alpha.$ Note that the order we multiply the matrices together is the same as the one we use in the proof of Lemma \ref{train track lemma} even though we apply $\psi$ instead of $\psi^{-1}.$ This is because here sides should be thought of as row vectors, so we are multiplying matrices from the right, while previously the branches of the train track had to be thought of as column vectors, so we multiplied matrices from the left.\\

    Finally, we need to apply $\phi^{-1}.$ This just permutes the sides of the polygon. The columns of the matrix $P$ encode the permutation according to $\phi.$ Hence, in order to encode $\phi^{-1}$ in terms of rows of the matrix, we need to multiply with $(P^{-1})^T.$ Since $P$ is a permutation matrix, its transpose is its inverse, so $(P^{-1})^T = P.$\\ 
    
    We conclude that $f_\gamma^{-1}$ maps sides of $P_\gamma$ to paths that can be homotoped to lie only on sides of $P_\gamma$ and that the rows of the matrix $V_{e_1} \cdot ... \cdot V_{e_k} \cdot P = V_\gamma$ encode how often each path $f_\gamma^{-1}(\alpha)$ passes over each side.   
\end{proof} 

Lemma \ref{lemma1} shows that the rows of the matrix $V_\gamma$ describe the image of the corresponding sides of $P_\gamma$ under $f_\gamma^{-1}.$ Note furthermore, that if we orient the sides as in Figure \ref{fig:2n-gon} then for any $\alpha,$ we have that $f_\gamma^{-1}(\alpha)$ traverses each side $\beta$ in the given orientation. This is because the gluing of the triangle for a Rauzy-Veech move preserves this orientation.

\begin{lemma}\label{lemma2}
    Let $\alpha$ be a side of $P_\gamma$ that is never a winner for any edge of $\gamma.$ Then the $\alpha$-row of $V_\gamma$ just has a single non-zero entry which is equal to $1.$
\end{lemma}

\begin{proof}
    A cut and paste move corresponding to a Rauzy-Veech move maps any side of the starting polygon that is not a winner to some side of the resulting polygon. Hence, if $\alpha$ is never a winner, the image of $\alpha$ after applying all the edges of $\gamma$ will be a single side $\beta.$ According to Lemma \ref{lemma1}, the $\alpha$-row of $V_\gamma$ is consequently $0$ in every entry except for $(\alpha, \beta)$ where it is $1.$
\end{proof}

Lemma \ref{lemma2} tells us that if $\alpha$ is never a winner in any step of the path $\gamma,$ then the side corresponding to it in $P_\gamma,$ resp. the path in $X_\gamma,$ gets mapped under $f_\gamma^{-1}$ onto a single other side, resp. path. This will be crucial in the computation of the upper bound of the stable curve graph translation length in Section \ref{proof section}, since sides with this property are easy to control.

\subsection{The flip move}

In order to define the pseudo-Anosovs $f_g$ below, apart from Rauzy-Veech moves, we will need to consider the so called \textit{flip move} $f$. Similarly to the Rauzy-Veech moves, given a labeled permutation, the flip move produces a new labeled permutation in the following way:

$$\begin{pmatrix}
    \pi_t^{-1}(1) & \pi_t^{-1}(2) & ... & \pi_t^{-1}(n)\\
    \pi_b^{-1}(1) & \pi_b^{-1}(2) & ... & \pi_b^{-1}(n)\\
\end{pmatrix} 
\xlongrightarrow{f} 
\begin{pmatrix}
    \pi_b^{-1}(n) & \pi_b^{-1}(n-1) & ... & \pi_b^{-1}(1)\\
    \pi_t^{-1}(n) & \pi_t^{-1}(n-1) & ... & \pi_t^{-1}(1)\\
\end{pmatrix}.$$

For example, we have 

$$\begin{pmatrix}
    A & C & B\\
    B & A & C\\
\end{pmatrix}
\xlongrightarrow{f} 
\begin{pmatrix}
    C & A & B\\
    B & C & A\\
\end{pmatrix}.$$

We can also give a geometric meaning to the flip move. Let $(\pi_t, \pi_b)$ be a labeled permutation and $P = P(\pi_t, \pi_b)$ the corresponding polygon. We define the flip move on $P$ as rotation of $P$ by $180$ degrees. This defines a polygon $P'$ which yields a new labeled permutation $(\pi'_t, \pi'_b)$ that is exactly the one obtained by the flip move on $(\pi_t, \pi_b).$\\

We define the \textit{labeled augmented Rauzy diagram} to be the labeled Rauzy diagram with extra edges corresponding to the flip move. We label these extra edges by $f.$ Note that the labeled augmented Rauzy diagram might have fewer connected components than the labeled Rauzy diagram. As before, given a path $\gamma$ in the labeled augmented Rauzy diagram, we say that $\gamma$ is allowed if starting point and endpoint define the same unlabeled permutation, and in an analogous way we obtain a corresponding mapping class $f_\gamma$ of $X_\gamma.$\\

We again assign a matrix $V_\gamma$ to $\gamma$ in the same way as before, where when we pass through an edge corresponding to a flip move we don't change our matrix. So, $V_\gamma$ is again a non-negative integer matrix.

\begin{example}
    Let $$(\pi_t, \pi_b) = \begin{pmatrix}
        A & B & C\\
        C & A & B
    \end{pmatrix}$$ and consider the path $\gamma = f$ in the labeled augmented Rauzy diagram starting at $(\pi_t, \pi_b)$ and consisting of a single flip move. Since
    $$\begin{pmatrix}
        A & B & C\\
        C & A & B
    \end{pmatrix} 
    \xlongrightarrow{f} 
    \begin{pmatrix}
    B & A & C\\
    C & B & A\\
    \end{pmatrix},$$ we see that starting and ending permutation of $\gamma$ define the same unlabeled permutation, namely $(1 \ 2 \ 3).$ Hence, we obtain a mapping class $f_\gamma$ and a corresponding matrix $V_\gamma.$ In this case, $V_\gamma$ is equal to the permutation matrix, since the flip move doesn't contribute anything by definition, i.e.
    $$V_\gamma = P = \begin{pmatrix}
        0 & 1 & 0\\
    1 & 0 & 0\\
    0 & 0 & 1
    \end{pmatrix}.$$
    The mapping class $f_\gamma$ is an involution of the torus, since in terms of polygons it corresponds to a 180 degree rotation of a hexagon with sides glued as imposed by the permutations.
\end{example}

In the above example, the matrix $V_\gamma$ was of order $2,$ so in particular not primitive. Here, we show:

\begin{lemma}\label{flip move lemma}
    Let $\gamma$ be a path in the labeled augmented Rauzy diagram. Let $f_\gamma$ and $V_\gamma$ be the corresponding mapping class and matrix. If $V_\gamma$ is a primitive matrix, then $f_\gamma$ is pseudo-Anosov.
\end{lemma}

\begin{proof} 
    Assume that $\gamma$ consists of $n$ edges out of which $1 \le k \le n$ are flip moves. The main observation for the proof is that geometrically a flip move followed by a Rauzy-Veech move is the same as an analogous Rauzy-Veech move on the left side of the polygon followed by a flip move. More precisely, if we denote cut and paste moves (as defined in Section \ref{top RV}) by $t^R, \, b^R$ and the analogously defined cut and paste moves on the left side of the polygon by $t^L, \, b^L,$ we have
    $$ft^R = b^Lf, \text{ and } fb^R = t^Lf.$$

    These equations should be read as saying: Applying a top (resp. bottom) move on the right followed by a flip move is the same as first applying a flip move followed by a bottom (resp. top) move on the left.\\

    Hence, we can write our cut and paste homeomorphism $\psi$ corresponding to $f_\gamma$ as a composition of cutting and pasting on the right and on the left, followed by $k$ flip moves. Since the flip moves are of order two, depending on the parity of $k,$ we are either left with one or with zero flip moves.\\

    Now, analogously to the proof of Lemma \ref{train track lemma}, we have to check that the train track $\tau_\gamma$ is carried by $f_\gamma$ and that the train track matrix is given by $V_\gamma.$ If we only have to do cut and paste moves on the right, we already know the above from Lemma \ref{train track lemma}. Assume we have to do a cut and paste move on the left. We look at the case where we have to do a top move $t^L,$ (the case $b^L$ is analogous). This means that our path in the labeled augmented Rauzy diagram had an edge corresponding to a $b^R$ move which we substituted by $f t^L f = f t^L f^{-1} = b^R.$ Hence, we have to check that a $t^L$ move with winner-loser pair $(\alpha, \beta)$ has the same effect as a $b^R$ move with winner-loser pair $(\alpha, \beta).$ The action on the train track by the latter is given by the matrix $V_{\alpha, \beta}.$\\ 
    
    Performing a top move on the left with winner-loser pair $(\alpha, \beta)$ means that we cut the triangle on the left and glue its $\alpha$-side to the corresponding $\alpha$-side on the bottom (compare Figure \ref{fig:left top move}, which is the flipped version of Figure \ref{fig:cut and paste}).\\
    
     \begin{figure}
        \centering
        \includegraphics[scale=0.7]{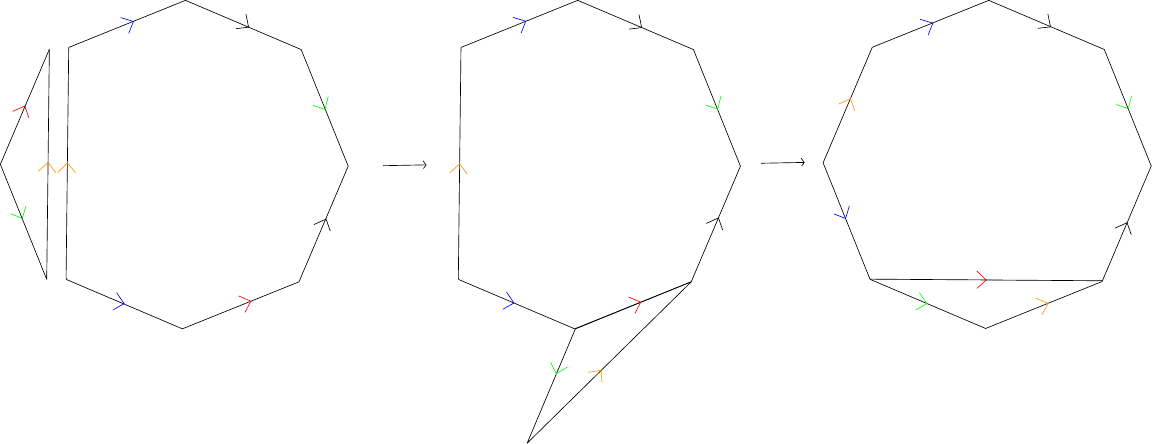}
        \caption{A top move on the left}
        \label{fig:left top move}
    \end{figure}
    
    In order to see the effect on the train track, we have to apply the inverse of this move and observe that every branch of the train track remains the same, except for the $\beta$-branch that now runs exactly once over itself and once over the $\alpha$-branch (compare Figure \ref{fig:train track effect for left move}, the flipped version of Figure \ref{fig:train track effect}, where $(\alpha, \beta) = (B,C)$). This is captured by the matrix $V_{\alpha, \beta}$ which is what we wanted to show.\\

    \begin{figure}[h]
    \centering
    \begin{tikzpicture}
    
        \node[anchor=south west,inner sep=0] at (0,0){\includegraphics[scale=0.7]{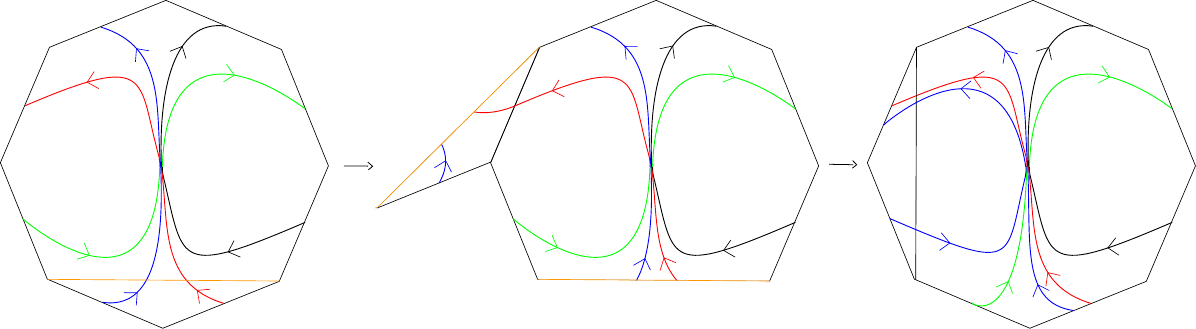}};
        \node at (0,2.8) {$B$};
        \node at (1.1, 3.8) {$C$};
        \node at (2.9, 3.8) {$A$};
        \node at (3.9,2.8) {$D$};
        \node at (0,1.2) {$D$};
        \node at (1.1, 0.1) {$C$};
        \node at (2.9, 0.1) {$B$};
        \node at (3.9,1.2) {$A$};

        \node at (5.25,2.6) {$\textcolor{orange}{B}$};
        \node at (7.75,0.35) {$\textcolor{orange}{B}$};

        \node at (10.3,2.8) {$B$};
        \node at (11.4, 3.8) {$C$};
        \node at (13.2, 3.8) {$A$};
        \node at (14.1,2.8) {$D$};
        \node at (10.3,1.2) {$C$};
        \node at (11.4, 0.1) {$D$};
        \node at (13.2, 0.1) {$B$};
        \node at (14.1,1.2) {$A$};
    \end{tikzpicture}
    
    \caption{The inverse of a left $t$-move with its effect on the train track}
    \label{fig:train track effect for left move}
\end{figure}

    Note, that if after performing all left and right cut and paste moves, $\psi$ does consist of a flip move, this only changes the orientation of the train track. The $\alpha$-branch of the train track will still be the $\alpha$-branch after the flip for every $\alpha$. Hence, the train track $\tau_\gamma$ is carried by $f_\gamma$ and the train track matrix is given by $V_\gamma.$\\

    It is a standard fact that if a large, birecurrent train track is carried by a mapping class and the corresponding train track matrix is primitive, then the mapping class is pseudo-Anosov (\cite{P1}, Corollary 3.2). The train track $\tau_\gamma$ is large, since its complement is a $(2n-2)$-gon. Furthermore, $\tau_\gamma$ is recurrent since every branch itself is a closed train route. Finally, each branch of $\tau_\gamma$ is intersected efficiently by a simple closed curve that can be built as a concatenation of sides of the corresponding polygon $P_\gamma.$ This concludes the proof.
\end{proof}
 
\section{Translation lengths in hyperelliptic components}\label{proof section}

In this section, we prove Theorems \ref{mytheorem} and \ref{bettertheorem}. We start by defining the pseudo-Anosov maps $f_g,$ which is due to Boissy--Lanneau. We are not going to prove that the $f_g$ are indeed in a hyperelliptic component, but instead refer the reader to \cite{BL1} or \cite{BL}. When defining the $f_g,$ we use a different labeling of the permutations Boissy--Lanneau use. We then proceed with proving Theorem \ref{mytheorem}. We split the proof into the proof for the upper bound and the lower bound of $l_\mathcal{C}(f_g)$ since the proofs use different techniques. Finally, we prove Theorem \ref{bettertheorem}.

\subsection{The pseudo-Anosovs $f_g$}

First, we recall the construction of the $f_g$ which can be found in \cite{BL1} as well as \cite{BL}.\\ 

Consider the labeled permutation

$$(\pi_t^g, \pi_b^g) = \begin{pmatrix}
    \alpha_1 & \alpha_2 & \alpha_3 & ... & \alpha_g & \alpha_{g+1} & \alpha_{g+2} & ... & \alpha_{2g}\\
    \alpha_{2g} & \alpha_{g-1} & \alpha_{g-2} & ... & \alpha_1 & \alpha_{2g-1} & \alpha_{2g-2} & ... & \alpha_g\\
\end{pmatrix}$$

and the path $\gamma_g = ftb^g$ in the labeled augmented Rauzy diagram starting at $(\pi_t^g, \pi_b^g),$ i.e. $\gamma_g$ is the path obtained by first applying $g$ times the type $b$ move, followed by once the type $t$ move and finally once the flip move $f$. Note that we write $ftb^g,$ i.e. use the convention to read from right to left. Furthermore, we write $x \cdot (\pi_t, \pi_b)$ to denote the endpoint of any path $x$ in the labeled augmented Rauzy-diagram starting at a labeled permutation $(\pi_t, \pi_b).$\\

After applying $b$ a number of $k$ times ($k \in \{1, ... , g\}$) to our starting permutation, we end up at the permutation

$$b^k \cdot (\pi_t^g, \pi_b^g) = \begin{pmatrix}
    \alpha_1 & \alpha_2 & ... & \alpha_g & \alpha_{2g-k+1} & ... & \alpha_{2g} & \alpha_{g+1} & ... & \alpha_{2g-k}\\
    \alpha_{2g} & \alpha_{g-1} & ... & \alpha_1 & \alpha_{2g-1} & ... & \alpha_{2g-k} & \alpha_{2g-k-1} & ... & \alpha_g\\
\end{pmatrix}.$$

Whenever we write "..." in the above permutation, we mean to continue the sequence by adding or subtracting one to the subscript as appropriate. We observe that after $g$ moves of type $b$ we obtain $b^g \cdot (\pi_t^g, \pi_b^g) = (\pi_t^g, \pi_b^g).$ Hence, we traced out a loop in the labeled augmented Rauzy diagram and can therefore already obtain a mapping class corresponding to this loop. However, this mapping class wouldn't be a pseudo-Anosov. We continue the path by applying $t.$\\

From the above, we obtain

$$tb^g \cdot (\pi_t^g, \pi_b^g) = t \cdot (\pi_t^g, \pi_b^g) = \begin{pmatrix}
    \alpha_1 & \alpha_2 & ... & \alpha_{g+1} & \alpha_{g+2} & ... & \alpha_{2g}\\
    \alpha_{2g} & \alpha_g & ... & \alpha_1 & \alpha_{2g-1} & ... & \alpha_{g+1}\\
\end{pmatrix}.$$

Finally, we apply $f$ to obtain

$$\gamma_g \cdot (\pi_t^g, \pi_b^g) = f \cdot \begin{pmatrix}
    \alpha_1 & \alpha_2 & ... & \alpha_{g+1} & \alpha_{g+2} & ... & \alpha_{2g}\\
    \alpha_{2g} & \alpha_g & ... & \alpha_1 & \alpha_{2g-1} & ... & \alpha_{g+1}\\
\end{pmatrix} = $$
$$ = \begin{pmatrix}
    \alpha_{g+1} & ... & \alpha_{2g-1} & \alpha_1 & ... & \alpha_g & \alpha_{2g}\\
    \alpha_{2g} & ... & \alpha_{g+2} & \alpha_{g+1} & ... & \alpha_2 & \alpha_1\\
\end{pmatrix}.$$\\[-0.5cm]

Comparing starting point and endpoint of $\gamma_g,$ we see that the two labeled permutations define the same unlabeled permutation, and therefore $\gamma_g$ defines a mapping class $f_g$ on the surface $X_g := X(\pi_t^g, \pi_b^g).$ One easily checks by induction that the surface $X_g$ is indeed a surface of genus $g.$\\

Boissy--Lanneau show that $f_g$ is pseudo-Anosov and that its stretch factor is bounded below by $\sqrt{2}$ for all $g$ (\cite{BL1}). This is done by computing the matrix $V_{\gamma_g},$ showing that it is primitive and computing its spectral radius. Since Boissy--Lanneau use a different labeling, the matrix they obtain differs from the one we compute below by conjugation by a permutation matrix. 

\subsection{Proof of Theorem \ref{mytheorem}}

Our main goal is to prove the following:

\begin{theorem*}
    For all $g \ge 2,$ it holds that
    $$\frac{1}{16g-12} \, \le \, l_C(f_g) \, \le \, \frac{1}{g-1}.$$
\end{theorem*}

The remaining part of this section is devoted to proving the theorem. Since $l_\mathcal{C}(f_g) = l_\mathcal{C}(f_g^{-1}),$ we can use either $f_g$ or its inverse for the proof. We use $f_g^{-1}$ for the proof of the upper bound and $f_g$ for the proof of the lower bound.\\

The first thing to note is that on the surface $X(\pi_t^g, \pi_b^g),$ all polygon sides $\alpha_1, ... , \alpha_{2g}$ are closed curves. This follows from the equivalent statement in the following lemma.

\begin{lemma}
    All the corners of the polygon $P(\pi_t^g, \pi_b^g)$ define the same point in $X(\pi_t^g, \pi_b^g).$
\end{lemma}

\begin{proof}
    Consider the corner on the left of the top side corresponding to $\alpha_1.$ Since in $X(\pi_t^g, \pi_b^g)$ top and bottom side corresponding to $\alpha_1$ are identified, the above corner defines the same point as the left corner of the bottom side corresponding to $\alpha_1.$ This in turn is the same corner as the right corner of the bottom side corresponding to $\alpha_2.$ From the gluing, we move to the right corner of the top side corresponding to $\alpha_2,$ which is the same as the left corner of the top side corresponding to $\alpha_3.$ Iterating this process shows that all the left corners of the top sides corresponding to $\alpha_i$ with $i$ odd define the same point in $X(\pi_t^g, \pi_b^g).$\\
    
    After reaching the left corner of the top side corresponding to $2g-1$ and gluing to the bottom, we end up at the right corner of the bottom side corresponding to $\alpha_1,$ which is the same as the left corner of the top side corresponding to $\alpha_2.$ From there, iterating again shows that we also identify all the left corners of the top sides of the $\alpha_i$ with $i$ even. Hence, all the corners get identified in $X(\pi_t^g, \pi_b^g).$
\end{proof}

Since the $\alpha_i$ are closed curves, we can use them to find bounds on the stable curve graph translation length of $f_g^{-1}.$ The ones that are easier to control are the ones that, seen as sides of the polygon, get mapped to another side of the polygon.

\begin{lemma}
    The only two curves that are winners for some edge of $\gamma_g$ are $\alpha_g$ and $\alpha_{2g}.$
\end{lemma}
\begin{proof}
    For the first $b$ move in $\gamma_g,$ the winner is $(\pi_b^g)^{-1}(2g) = \alpha_g.$ Since after a $b$ move, $\pi_b^g$ remains unchanged, we get that for all the following $b$ moves, the winner is still $\alpha_g.$ Finally, we apply the $t$ move to $b^{g} \cdot (\pi_t^g, \pi_b^g) = (\pi_t^g, \pi_b^g),$ so the winner is $(\pi_t^g)^{-1}(2g) = \alpha_{2g}.$ 
\end{proof}

The above, together with Lemma \ref{lemma2}, implies that all other curves $\alpha_i$ with $i \notin \{g, 2g\}$ are mapped by $f_g^{-1}$ onto some other $\alpha_j,$ where $j \in \{1, ... ,2g\}$. Therefore, for any curve $\alpha_i$ with $i \notin \{g, 2g\},$ we can read off its image just by looking at the labeled permutation that is the endpoint of $\gamma.$\\

In particular, if we write $f_g^{-1} = \phi^{-1} \circ \psi$ as in Section \ref{construction}, then for $i \notin \{g, 2g\},$ $\psi$ takes $\alpha_i$ to the side $\alpha_i$ of $X(\gamma_g(\pi_t^g, \pi_b^g)).$ Hence, $f_g(\alpha_i) = \phi^{-1}(\alpha_i).$ It remains to study the relabeling homeomorphism $\phi.$\\

Let $(\widetilde{\pi}_t^g, \widetilde{\pi}_b^g) := \gamma_g \cdot (\pi_t^g, \pi_b^g)$ be the endpoint of $\gamma_g.$ Since on the level of polygons $\phi$ just maps $P(\pi_t^g, \pi_b^g)$ onto $P(\widetilde{\pi}_t^g, \widetilde{\pi}_b^g)$ as the identity, we can read off what $\phi$ does on the sides of the polygons by just looking at the labeled permutations, i.e.
$$\phi(\alpha_i) = (\widetilde{\pi}_t^g)^{-1} \circ \pi_t^g (\alpha_i), \text{ or equivalently } \phi(\alpha_i) = (\widetilde{\pi}_b^g)^{-1} \circ \pi_b^g (\alpha_i).$$ 

We summarise what we showed in the following lemma.

\begin{lemma}\label{goodcurves}
    For any $i \notin \{g, 2g\},$ the pseudo-Anosov $f_g^{-1}$ maps the curve $\alpha_i$ to the curve 
    $$\phi^{-1}(\alpha_i) = (\pi_t^g)^{-1} \circ \widetilde{\pi}_t^g(\alpha_i).$$ 
\end{lemma}

From the starting and ending points of $\gamma,$ we read off that

$$(\pi_t^g)^{-1} \circ \widetilde{\pi}_t^g(\alpha_i) = \begin{cases}
   \alpha_{g+i-1} & \text{if } \, i = 1, ... ,g\\
   \alpha_{i-g} & \text{if } \, i = g+1, ... , 2g-1\\
   \alpha_{2g} & \text{if } \, i = 2g.\\
\end{cases}$$

We are now ready to prove the upper bound for Theorem \ref{mytheorem}.

\begin{proof} (Upper bound in Theorem \ref{mytheorem})\\

The strategy we want to use for proving that a pseudo-Anosov $f$ satisfies $l_\mathcal{C}(f) \le \frac{k}{n}$ for some $k , n \in \mathbb{N}$ is the following: We find a curve $\alpha$ such that $$d_\mathcal{C}(\alpha, f^n(\alpha)) \le k.$$ 
Then, by the triangle inequality and the fact that $f$ acts as an isometry on the curve graph, it follows that for any $j \ge 1$
$$\frac{d_\mathcal{C}(\alpha, f^{jn}(\alpha))}{j} \le \frac{j \cdot d_\mathcal{C}(\alpha, f^n(\alpha))}{j} \le k,$$
and so $l_\mathcal{C}(f^n) \le k.$ It is now easy to see that then $l_\mathcal{C}(f) \le \frac{k}{n}$ (see \cite{Vaibhav}, Lemma 2.2).\\

Consider $f_g$ as defined in the beginning of the section. We want to apply the above mentioned strategy to $f_g^{-1}$ and the curve $\alpha_{2g-1}.$\\

From Lemma \ref{goodcurves} and the equation after, we see that $f_g^{-1}(\alpha_{2g-1}) = \alpha_{g-1}.$ Iterating this process yields

\begin{align*}
    f_g^{-1}(\alpha_{2g-1}) = \alpha_{g-1}\\
    f_g^{-2}(\alpha_{2g-1}) = \alpha_{2g-2}\\
    f_g^{-3}(\alpha_{2g-1}) = \alpha_{g-2}\\
    f_g^{-4}(\alpha_{2g-1}) = \alpha_{2g-3}\\
    \vdots\\
    f_g^{-(2g-2)}(\alpha_{2g-1}) = \alpha_g.
\end{align*} 

More formally, we have $f_g^{-k}(\alpha_{2g-1}) = \alpha_{g-\frac{k+1}{2}}$ for all odd $k$ between $1$ and $2g-3$ and $f_g^{-k}(\alpha_{2g-1}) = \alpha_{2g-\frac{k+2}{2}}$ for all even $k$ between $2$ and $2g-2.$\\

Note that after $2g-2$ iterates of $f_g^{-1}$, we hit the curve $\alpha_g$ for which Lemma \ref{goodcurves} doesn't apply anymore, so we can't use it to say something about $f_g^{-(2g-1)}(\alpha_{2g-1}).$ Note furthermore, that the orbit of $\alpha_{2g-1}$ under $f_g^{-1}$ traces out all the curves $\alpha_i$ with $i \notin \{g, 2g\}$ before hitting the curve $\alpha_g,$ so out of all the choices we could make in order to apply Lemma \ref{goodcurves}, this gives the best upper bound.\\

Since $\alpha_{2g-1}$ and $\alpha_g$ are both sides of the polygon $P(\pi_t^g, \pi_b^g),$ the curves $\alpha_{2g-1}$ and $\alpha_g$ of the surface $X(\pi_t^g, \pi_b^g)$ intersect in one point, which is the image of the corners of $P(\pi_t^g, \pi_b^g).$ So, the geometric intersection number of the two curves is at most $1$.\\

If the geometric intersection number is $0,$ then $d_\mathcal{C}(\alpha_{2g-1}, \alpha_g) = 1.$ If the geometric intersection number is $1,$ then by thickening up the curves, one can find a neighbourhood of their union that is homeomorphic to a torus with one boundary component $\beta$. This $\beta$ is a closed curve that is disjoint from both $\alpha_{2g-1}$ and $\alpha_g,$ and it is furthermore essential, because the genus of our surface is $\ge 2.$ It follows that $d_\mathcal{C}(\alpha_{2g-1}, \alpha_g) = 2.$\\

In any case, we have 
$$d_\mathcal{C}(\alpha_{2g-1}, f_g^{-(2g-2)}(\alpha_{2g-1})) = d_\mathcal{C}(\alpha_{2g-1}, \alpha_g) \le 2.$$

It follows that
$$l_\mathcal{C}(f_g) = l_\mathcal{C}(f_g^{-1}) \le \frac{2}{2g-2} = \frac{1}{g-1}.$$
 
\end{proof}

It remains to prove the lower bound of Theorem \ref{mytheorem}. This requires different techniques from the ones used in the proof of the upper bound. It is done by using an invariant train track for $f_g$ and the nesting lemma presented in Section \ref{trains}.

\begin{proof}(Lower bound in Theorem \ref{mytheorem})\\

    Let $\tau$ be the train track on $X_g$ as defined in Section \ref{properties section} (compare Figure \ref{fig:train track}). The train track $\tau$ is large, since its compliment is a $(4g-2)$-gon. Furthermore, $\tau$ is recurrent since every branch itself is a closed train route. Finally, in Lemma \ref{train track lemma} (and the proof of Lemma \ref{flip move lemma}) we showed that $\tau$ is invariant for $f_g$ and the corresponding train track matrix is given by $V_{\gamma_g}.$\\ 

    We have to compute $V_{\gamma_g}.$ Since the sequence of winner-loser pairs of $\gamma_g$ is given by $$(\alpha_g, \alpha_{2g}), \, (\alpha_g, \alpha_{2g-1}), ... , (\alpha_g, \alpha_{g+1}), \, (\alpha_{2g}, \alpha_g),$$ we obtain
    
    $$V_{\gamma_g} = V_{\alpha_{g}, \alpha_{2g}} \cdot V_{\alpha_g, \alpha_{2g-1}} \cdot ... \cdot V_{\alpha_g, \alpha_{g+1}} \cdot V_{\alpha_{2g}, \alpha_g} \cdot P,$$

    where $P$ is the permutation matrix as defined in Section \ref{construction}.\\ 

    With the convention that the $\alpha_i$ are ordered according to their indices, we compute
    
    $$V_{\alpha_{g}, \alpha_{2g}} \cdot V_{\alpha_g, \alpha_{2g-1}} \cdot ... \cdot V_{\alpha_g, \alpha_{g+1}} = \begin{pmatrix}
        Id & A\\
        0 & Id\\
    \end{pmatrix},$$

    where each block is a $g \times g$ block and $A$ is a matrix whose entries in the last row are all $1$ and every other entry is $0.$ Furthermore, we have $$V_{\alpha_{2g}, \alpha_g} = \begin{pmatrix}
        Id & 0\\
        B & Id\\
    \end{pmatrix},$$ where $B$ has a single non-zero entry equal to $1$ in its bottom right corner and $$P = \begin{pmatrix}
        0_{g \times g-1} & Id_{g \times g} & 0_{g \times 1}\\
        Id_{g-1 \times g-1} & 0_{g-1 \times g} & 0_{g-1 \times 1}\\
        0_{1 \times g-1} & 0_{1 \times g} & 1_{1 \times 1}\\
    \end{pmatrix},$$

    where the subscripts indicate the size of each block.\\

    In total, we obtain

    $$V_{\gamma_g} = \begin{pmatrix}
        Id & A\\
        0 & Id\\
    \end{pmatrix} \begin{pmatrix}
        Id & 0\\
        B & Id\\
    \end{pmatrix} P = \begin{pmatrix}
        Id + B & A\\
        B & Id\\
    \end{pmatrix} P = \begin{pmatrix}
        A_{g \times g-1} & Id_{g \times g} + B_{g \times g} & B_{g \times 1}\\
        Id_{g-1 \times g-1} & 0_{g-1 \times g} & 0_{g-1 \times 1}\\
        0_{1 \times g-1} & B_{1 \times g} & 1_{1 \times 1}
    \end{pmatrix}.$$

    In the above, the subscripts indicate the size of the matrices. $A_{g \times g-1}$ is a matrix whose entries in the last row are all $1$ and every other entry is $0.$ $B_{g \times g}$ is the same as $B$ and $B_{g \times 1}, B_{1 \times g}$ have all entries $0$ except for the $(g, 1)$ or $(1,g)$ entry respectively, which is equal to $1.$\\ 

    Since the $(2g,2g)$ entry on the diagonal of $V_{\gamma_g}$ is positive, we can use (\cite{Tsai}, Proposition 2.4) to argue that $V_{\gamma_g}^{4g}$ is a positive matrix. In fact, a more thorough inspection of $V_{\gamma_g}$ shows that already its $(4g-4)^{\text{th}}$ power is a positive matrix, but we omit the details since we only need to obtain a bound of order $\frac{1}{g}$.\\

    The fact that $V_{\gamma_g}^{4g}$ is a positive matrix implies that given a curve $\gamma$ carried by $\tau,$ the curve $f_g^{4g}(\gamma)$ passes over every branch of $\tau.$ Given instead a curve $\gamma$ carried by some diagonal extension of $\tau,$ then (\cite{Vaibhav}, Lemma 5.2) implies that $f_g^{K}(\gamma)$ with $K = 6(2g-2) + 4g = 16g - 12$ passes over every branch of $\tau.$\\

    Now, for $n \ge 1$ let $\tau_n := f_g^{Kn}(\tau)$ and set $\tau = \tau_0.$ Then, the above observation shows that $$PE(\tau_n) \subset int(PE(\tau_{n-1})).$$ From the Nesting Lemma (Section \ref{trains}) we obtain the following sequence of inclusions:
    $$... \, PE(\tau_n) \subset int(PE(\tau_{n-1})) \subset \mathcal{N}_1(int(PE(\tau_{n-1}))) \subset PE(\tau_{n-1}) \subset int(PE(\tau_{n-2})) ...$$

    Finally, choose a curve $\delta$ which is not carried by any extension of $\tau$ but such that $\gamma := f_g^{K}(\delta)$ is. Then for $n \ge 1,$ we get that $f_g^{Kn}(\gamma)$ is an element of $PE(\tau_n).$ The above sequence of inclusions now implies that $d_{\mathcal{C}}(\gamma, f_g^{Kn}(\gamma)) \ge n$ for any $n \ge 1.$ We conclude:

    $$l_\mathcal{C}(f_g^K) = \liminf\limits_{n \to \infty} \frac{d_\mathcal{C}(\gamma, f_g^{Kn}(\gamma))}{n} \ge 1,$$

    which implies
    
    $$l_\mathcal{C}(f_g) = \frac{l_\mathcal{C}(f_g^K)}{K} \ge \frac{1}{K} = \frac{1}{16g-12}.$$
\end{proof}

\subsection{Proof of Theorem \ref{bettertheorem}}

This Section is devoted to the proof of Theorem \ref{bettertheorem}.\\

Let $g \ge 2$ and 
$$l_g := \min \, \{ \,  l_\mathcal{C}(f) \, | \, f \text{ is a pseudo-Anosov in a hyperelliptic component of genus } g\}.$$ In order to prove Theorem \ref{bettertheorem}, we have to prove that $l_g \asymp \frac{1}{g}.$ Hence, it suffices to find an upper and a lower bound for $l_g$ that are of order $\frac{1}{g}.$ From Theorem \ref{mytheorem}, together with the fact that the pseudo-Anosov $f_g$ is in a hyperelliptic component, it follows immediately that $l_g \le l_\mathcal{C}(f_g) \le \frac{1}{g-1}.$ Hence, it remains to find a lower bound. For that, we will show that any pseudo-Anosov in a hyperelliptic component of genus $g$ has stable curve graph translation length greater or equal to $\frac{1}{16g-10}.$ The argument we use is very similar to the proof of the lower bound in Theorem \ref{mytheorem}. We start with some preliminary observations that are necessary for the proof.\\

Let $f$ be a pseudo-Anosov in a hyperelliptic component. Then, there exists a translation surface $(X, \omega)$ in a hyperelliptic component, a representative $\phi$ of $f$ and a corresponding pseudo-Anosov $\phi_X$ on $X$ (compare Section \ref{affine pA}). Note that $l_\mathcal{C}(f) = l_\mathcal{C}(\phi_X).$ Since $X$ is a hyperelliptic Riemann surface, it admits a hyperelliptic involution $\tau.$ Following the notation of Boissy--Lanneau, we let $\{\phi_X, \tau \circ \phi_X\} = \{\phi^+, \phi^-\},$ where $\phi^+$ preserves the orientation of the associated foliations and $\phi^-$ reverses it (see \cite{BL} for details). The important property for our purposes is given by the following lemma:

\begin{lemma}
    It holds that $l_\mathcal{C}(\phi^+) = l_\mathcal{C}(\phi^-).$
\end{lemma}
\begin{proof}
    We have to show that $l_\mathcal{C}(\phi_X) = l_\mathcal{C}(\tau \circ \phi_X).$ Since $\tau$ is a hyperelliptic involution, there is a curve $\alpha$ that is fixed by $\tau.$ Furthermore, conjugating $\tau$ by $\phi_X$ yields a hyperelliptic involution of $X$ which -since $\tau$ is unique- has to be equal to $\tau.$ Thus, $\tau$ commutes with $\phi_X$ (cf. \cite{BL1}, Lemma 2.3). It follows that
    $$(\tau \circ \phi_X)^n(\alpha) = \phi_X^n \circ \tau^n(\alpha) = \phi_X^n(\alpha)$$
    for any $n \in \mathbb{N},$ and therefore 
    $$l_\mathcal{C}(\tau \circ \phi_X) = \liminf\limits_{n \to \infty} \frac{d_\mathcal{C}(\alpha, (\tau \circ \phi_X)^n(\alpha))}{n} = \liminf\limits_{n \to \infty} \frac{d_\mathcal{C}(\alpha, \phi_X^n(\alpha))}{n} = l_\mathcal{C}(\phi_X).$$
\end{proof}

Let $n \in \{2g, \, 2g+1\}.$ Let $\mathcal{A}$ be an alphabet consisting of $n$ letters $\alpha_1, ... , \alpha_n$ and let 
$$(\pi_t, \pi_b) = \begin{pmatrix}
    \alpha_1 & \alpha_2 & ... & \alpha_{n-1} & \alpha_n\\
    \alpha_n & \alpha_{n-1} & ... & \alpha_2 & \alpha_1
\end{pmatrix}.$$

Throughout this section, we refer to $(\pi_t, \pi_b)$ as the \textit{central permutation}. Consider the connected component of the labeled Rauzy diagram of $(\pi_t, \pi_b).$ We call this component $\mathcal{D}_{(\pi_t, \pi_b)}.$ Rauzy showed that this connected component is isomorphic to the so called unlabeled Rauzy diagram of $\pi = \pi_b \circ \pi_t^{-1}$ (see \cite{R} or \cite{BL1}). This means in particular that no two different labeled permutations in $\mathcal{D}_{(\pi_t, \pi_b)}$ define the same unlabeled permutation.\\ 

Any mapping class defined through Rauzy-Veech induction on $\mathcal{D}_{(\pi_t, \pi_b)}$ is affine for a hyperelliptic component. More concretely, if $n = 2g,$ then it's affine for $\mathcal{H}(2g-2),$ and if $n = 2g+1,$ then for $\mathcal{H}(g-1, g-1)$ (see for example \cite{BL1}).\\

Given the central permutation, we can perform $t$-moves to it and obtain a loop in $\mathcal{D}_{(\pi_t, \pi_b)}$ consisting of $n-1$ vertices corresponding to 
$$t^m \cdot (\pi_t, \pi_b) = \begin{pmatrix}
    \alpha_1 & \alpha_2 & \alpha_3 & ... & \alpha_{m+1} & \alpha_{m+2} & \alpha_{m+3} & ... & \alpha_n\\
    \alpha_n & \alpha_{m} & \alpha_{m-1} & ... & \alpha_1 & \alpha_{n-1} & \alpha_{n-2} & ... & \alpha_{m+1}
\end{pmatrix}$$
for $m = 1, ... , n-1.$ We refer to this loop in the labeled Rauzy diagram as the \textit{central loop}.\\

We are now ready to prove Theorem \ref{bettertheorem}:

\begin{proof}(Theorem \ref{bettertheorem})\\

    Let $f$ be in a hyperelliptic component of genus $g$ and let $\phi^+, \phi^-$ be as constructed above. In order to prove Theorem \ref{bettertheorem}, it suffices to show that $l_\mathcal{C}(\phi^+) = l_\mathcal{C}(\phi^-) \ge \frac{1}{16g-10}.$\\

    Boissy--Lanneau prove the following statement:\\

    If $\phi^+$ cannot be obtained by Rauzy-Veech induction on $\mathcal{D}_{(\pi_t, \pi_b)},$ then $\phi^-$ is obtained by a path in the augmented labeled Rauzy diagram that starts at a permutation in the central loop and consists of a single flip move which corresponds to the last edge of the path (cf. \cite{BL}, Theorem 4.1).\\

    From this, we obtain that at least one of $\phi^+, \phi^-$ will be of the form $f_\gamma$ (compare the notation of Section \ref{top RV}) where $\gamma$ is one of two cases:
    \begin{itemize}
        \item $\gamma$ is a closed path in $\mathcal{D}_{(\pi_t, \pi_b)}.$
        \item $\gamma$ is a path in the augmented labeled Rauzy diagram starting at some $t^m \cdot (\pi_t, \pi_b)$ and whose last edge is the only edge with label $f.$ 
    \end{itemize}

    Hence, it suffices to analyze the stable curve graph translation length of $f_\gamma$ in the two cases.\\
    
    Recall from Section \ref{properties section} that we assign a matrix $V_\gamma$ to $f_\gamma$ which is a train track matrix for a natural train track carried by $f_\gamma.$ We want to have a power $k$ such that $V_\gamma^k$ becomes a positive matrix. Having found such a $k,$ we can use (\cite{Vaibhav}, Lemma 5.2) and argue analogously to the proof of the lower bound in Theorem \ref{mytheorem} to conclude $l_\mathcal{C}(f_\gamma) \ge \frac{1}{6(2g-2) + k}.$\\

    We are left with the task of finding the power $k.$ Recall, that since $V_\gamma$ is primitive, it suffices to find a non-zero diagonal entry. Then, we can apply the result of Tsai (\cite{Tsai}, Proposition 2.4) to obtain $k$.\\

    We analyze the two cases of the path $\gamma$ separately:\\
    
    In the first case, $\gamma$ is a closed path in $\mathcal{D}_{(\pi_t, \pi_b)}.$ So, $V_\gamma$ is a product of matrices $V_{\alpha, \beta}.$ In particular, we don't have to multiply with a permutation matrix at the end, because starting point and endpoint of $\gamma$ are the same labeled permutation. Since the $V_{\alpha, \beta}$ are non-negative matrices with all diagonal entries equal to $1,$ all diagonal entries of $V_\gamma$ are positive.\\

    In the second case, $\gamma$ is a path that starts at $t^m \cdot (\pi_t, \pi_b)$ for some $m \in \{1, ... , n-1\}.$ The edges of $\gamma$ are all $t$ or $b$ edges, except for the last one which is an $f$ edge. The last vertex of $\gamma$ defines the same unlabeled permutation as the starting point $t^m \cdot (\pi_t, \pi_b).$ This implies that the second to last vertex of $\gamma$ is a labeled permutation in $\mathcal{D}_{(\pi_t, \pi_b)}$ which defines the same unlabeled permutation as $f \cdot t^m \cdot (\pi_t, \pi_b).$ Since
    $$t^m \cdot (\pi_t, \pi_b) = \begin{pmatrix}
    \alpha_1 & \alpha_2 & \alpha_3 & ... & \alpha_{m+1} & \alpha_{m+2} & \alpha_{m+3} & ... & \alpha_n\\
    \alpha_n & \alpha_{m} & \alpha_{m-1} & ... & \alpha_1 & \alpha_{n-1} & \alpha_{n-2} & ... & \alpha_{m+1}
    \end{pmatrix},$$
    we have 
    $$f \cdot t^m \cdot (\pi_t, \pi_b) = \begin{pmatrix}
    \alpha_{m+1} & ... & \alpha_{n-2} & \alpha_{n-1} & \alpha_1 & ... & \alpha_{m-1} & \alpha_m & \alpha_n\\
    \alpha_n & ... & \alpha_{m+3} & \alpha_{m+2} & \alpha_{m+1} & ... & \alpha_3 & \alpha_2 & \alpha_1
    \end{pmatrix}.$$
    The labeled permutation 
    $$(\pi_t', \pi_b') := t^{n-m-1} \cdot (\pi_t, \pi_b) = \begin{pmatrix}
        \alpha_1 & \alpha_2 & \alpha_3 & ... & \alpha_{n-m} & \alpha_{n-m+1} & \alpha_{n-m+2} & ... & \alpha_n\\
        \alpha_n & \alpha_{n-m-1} & \alpha_{n-m-2} & ... & \alpha_1 & \alpha_{n-1} & \alpha_{n-2} & ... & \alpha_{n-m}
    \end{pmatrix}$$
    is an element of $\mathcal{D}_{(\pi_t, \pi_b)}$ which defines the same unlabeled permutation as $f \cdot t^m \cdot (\pi_t, \pi_b).$ Since there is a unique such permutation, the second to last vertex of $\gamma$ has to be $(\pi_t', \pi_b').$ Hence, the endpoint of $\gamma$ is
    $$f \cdot (\pi_t', \pi_b') = \begin{pmatrix}
        \alpha_{n-m} & ... & \alpha_{n-2} & \alpha_{n-1} & \alpha_1 & ... & \alpha_{n-m-2} & \alpha_{n-m-1} & \alpha_n\\
        \alpha_n & ... & \alpha_{n-m+2} & \alpha_{n-m+1} & \alpha_{n-m} & ... & \alpha_3 & \alpha_2 & \alpha_1
    \end{pmatrix}.$$
    From the above, we obtain that $V_\gamma$ is a product of matrices $V_{\alpha, \beta}$ followed by multiplication with a permutation matrix $P$ which encodes the relabeling between $t^m \cdot (\pi_t, \pi_b)$ and $f \cdot (\pi_t', \pi_b').$ Notice that the $n^{\text{th}}$ entry in the top row of the two labeled permutations is the same, namely $\alpha_n$. This implies that the $(n,n)$-entry of $P$ is equal to $1.$ Since any product of matrices $V_{\alpha, \beta}$ has positive entries on the diagonal, it follows that $V_\gamma$ has a positive entry on the diagonal, namely the entry $(n,n).$\\
     
    We showed that in both cases the matrix $V_\gamma$ has a positive entry on the diagonal. Since $V_\gamma$ is an $n \times n$ matrix with either $n = 2g$ or $n = 2g+1$, using (\cite{Tsai}, Proposition 2.4), we obtain that $V_\gamma^{4g+2}$ is a positive matrix. Hence, we have found the power $k = 4g+2$ and conclude that $l_\mathcal{C}(f_\gamma) \ge \frac{1}{6(2g-2) + 4g+2} = \frac{1}{16g - 10}.$\\

    Since this lower bound is of order $\frac{1}{g},$ this concludes the proof.
\end{proof}

\section{Proof of Theorem \ref{myothertheorem}}\label{pennerexample}

In this section, we prove Theorem \ref{myothertheorem} as well as Corollaries \ref{cor1} and \ref{cor2} which follow directly from the Theorem. For the proof of Theorem \ref{myothertheorem}, we first construct the $f_n$ and show that they are pseudo-Anosov. We proceed with finding bounds for their stretch factor and stable curve graph translation length respectively, in order to show the claims of the Theorem. The idea is to slighlty modify Penner's famous example of a pseudo-Anosov in order to obtain the desired properties.\\

For $g \ge 3,$ consider a genus $g$ surface $S$ as in Figure \ref{fig:surface}.\\ 

\begin{figure}[h]
    \centering
    \begin{tikzpicture}
        \node[anchor=south west,inner sep=0] at (0,0){\includegraphics[scale=0.4]{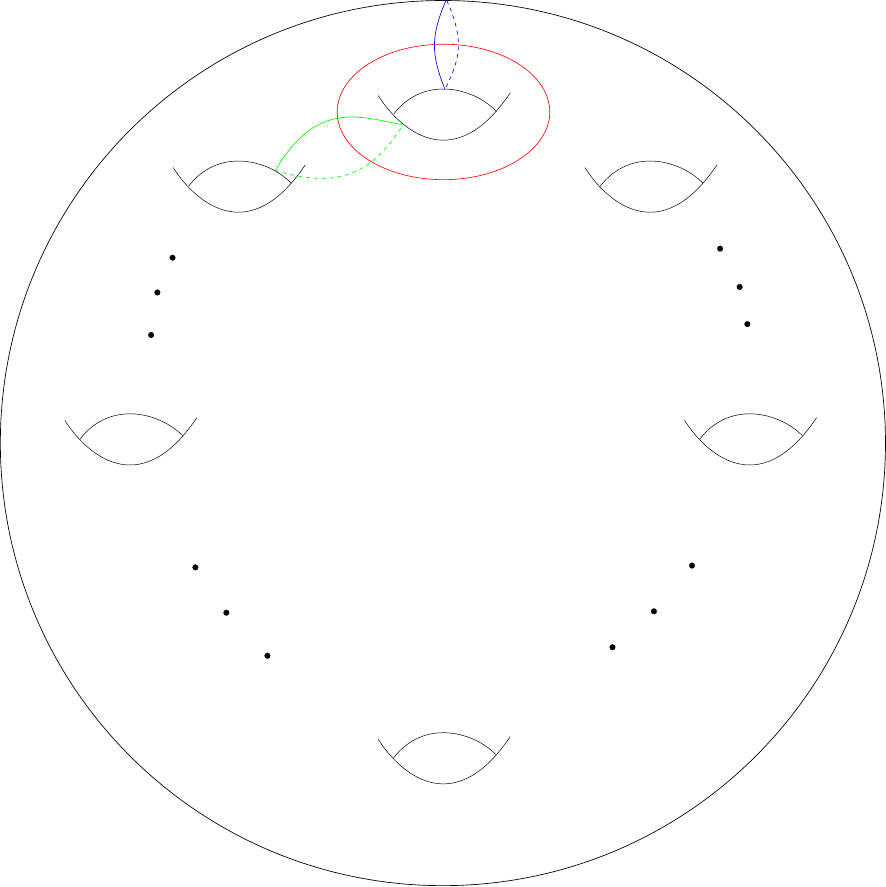}};
        \node at (3,6.2) {$\textcolor{blue}{b}$};
        \node at (2.25, 4.65) {$\textcolor{green}{c}$};
        \node at (3.1, 4.6) {$\textcolor{red}{a}$};
    \end{tikzpicture}
    \caption{Genus $g$ surface}
    \label{fig:surface}
\end{figure}

Let $\rho$ be the order $g$ homeomorphism that rotates the surface in Figure \ref{fig:surface} anti-clockwise and let $a, b, c$ be the curves depicted. Let $T_\gamma$ denote the right-handed Dehn twist about a curve $\gamma.$ Penner shows that the mapping class $\rho \circ T_c \circ T_a^{-1} \circ T_b$ is pseudo-Anosov (see \cite{P2}). We define $f_n \in \Mod(S)$ by $f_n := \rho \circ T_c \circ T_a^{-1} \circ T_b^{n}.$ 

\begin{lemma}
    $f_n$ is pseudo-Anosov.
\end{lemma}
\begin{proof}
    The curves $\rho^i(a), \rho^i(b), \rho^i(c)$ for $i = 1, ... , g$ fill the surface. Note that 
    $$f_n^g = (\rho \circ T_c \circ T_a^{-1} \circ T_b^{n})^g = \prod_{i = 1}^{g} T_{\rho^i(c)} \circ T_{\rho^i(a)}^{-1} \circ T_{\rho^i(b)}^{n}.$$ 
    From (\cite{P1}, Theorem 3.1), with $\mathcal{C} = \{\rho^i(b), \rho^i(c) \, | \, i = 1, ..., g\}$ and $\mathcal{D} = \{\rho^i(a) \, | \, i = 1, ..., g\},$ it follows that $f_n^g$ is pseudo-Anosov. Assume $f_n$ were periodic or reducible. Then, some power $f_n^m$ of $f_n$ fixes a curve. Hence $f_n^{mg}$ fixes that curve too, which is absurd since $f_n^g$ is pseudo-Anosov. So, $f_n$ has to be pseudo-Anosov. 
\end{proof}

\begin{lemma}\label{sf}
    For all $n,$ it holds that $\lambda(f_n) \ge n^{\frac{1}{g}},$ and consequently $l_\mathcal{T}(f_n) \ge \log(n^{\frac{1}{g}}).$
\end{lemma}

\begin{proof}
    The proof is analogous to the one Penner presents (\cite{P2}, Section "An upper bound by example"). There, Penner defines a train track which is invariant for $f_n.$ By analyzing the space of measures arising as counting measures of the curves 
    $$a, b, c, \rho(a), \rho(b), \rho(c), ... , \rho^{g-1}(a), \rho^{g-1}(b), \rho^{g-1}(c)$$ 
    on this train track, Penner constructs a $3g \times 3g$-matrix $M$ which counts how often each curve passes over any other curve and shows that the spectral radius of $M$ is the stretch factor of the pseudo-Anosov $f_1.$ More precisely, if we denote the above curves in the above order by $\gamma_i,$ with $i = 1, ... , 3g-3,$ then the $(i,j)$ entry of $M$ counts how often $f_1(\gamma_j)$ passes over $\gamma_i,$ i.e. the $j^{\text{th}}$ column describes the image of $\gamma_j.$\\
    
    Using the same argument, we obtain that $\lambda(f_n)$ can be computed as the spectral radius of the matrix
    $$M_n = \begin{pmatrix}
        0 & 0 & 0 & \cdots & 0 & Id\\
        A_n & B_n & 0 & \cdots & 0 & C_n\\
        0 & Id & 0 & \cdots & 0 & 0\\
        0 & 0 & Id & \cdots & 0 & 0\\
        \vdots & \vdots & \vdots & & \vdots & \vdots\\
        0 & 0 & 0 & \cdots & Id & 0\\
    \end{pmatrix},$$
    where all blocks are $3 \times 3$ matrices and $M_n$ is a $3g \times 3g$ matrix. $Id$ denotes the identity matrix and 
    $$
    A_n = \begin{pmatrix}
        n + 1 & 1 & 1\\
        n & 1 & 0\\
        n + 1 & 1 & 2\\
    \end{pmatrix}, \
    B_n = \begin{pmatrix}
        0 & 0 & 0\\
        0 & 0 & 0\\
        1 & 0 & 0\\
    \end{pmatrix}, \
    C_n = \begin{pmatrix}
        0 & 0 & 1\\
        0 & 0 & 0\\
        0 & 0 & 1\\
    \end{pmatrix}.
    $$
    Penner's example is exactly the case where $n = 1.$ Computing the $g^\text{th}$ iterate of the matrix for $g \ge 4 $ yields
    $$ M_n^g = \begin{pmatrix}
        A_n & B_n & 0 & 0 & \cdots & 0 & 0 & 0 & C_n\\
        C_nA_n & D_n + C_nB_n & B_nA_n & 0 & \cdots & 0 & 0 & 0 & C_n^2\\
        0 & C_n & D_n & B_nA_n & \cdots & 0 & 0 & 0 & 0\\
        \vdots & \vdots & \vdots & \vdots & \vdots & \vdots & \vdots & \vdots & \vdots\\
        0 & 0 & 0 & 0 & \cdots & 0 & C_n & D_n & B_nA_n\\
        B_nA_n & 0 & 0 & 0 & \cdots & 0 & 0 & C_n & D_n\\
    \end{pmatrix},
    $$
    where $D_n := A_n + B_nC_n.$ This computation is analogous to (\cite{P2}).\\
    
    In the case $g=3,$ one obtains

    $$M_n^g = \begin{pmatrix}
        A_n & B_n & C_n\\
        C_nA_n & D_n+C_nB_n & B_nA_n + C_n\\
        B_nA_n & C_n & D_n\\
    \end{pmatrix}.$$\\
    By the Collatz-Wielandt formula, we obtain that the spectral radius of $M_n^g$ is at least as big as the lowest row sum, which in both cases is given by the second row and equals $n + 1.$ Denoting the spectral radius of $M_n$ by $\rho(M_n),$ we have
    $$\lambda(f_n) = \rho(M_n) = \rho(M_n^g)^{\frac{1}{g}} \ge (n + 1)^{\frac{1}{g}} \ge n^{\frac{1}{g}}.$$
\end{proof}

\begin{lemma}\label{tl}
    For any $n,$ it holds that $l_\mathcal{C}(f_n) \le \frac{1}{g-1}.$
\end{lemma}

\begin{proof}
    Consider the curve $\gamma := \rho(b).$ Since it is disjoint from $a, b, c,$ we have $f_n(\gamma) = \rho(\gamma) = \rho^2(b).$ Iteratively, we obtain 
    $$f_n^{g-1}(\gamma) = \rho^{g-1}(\gamma) = \rho^{g}(b) = b.$$ 
    Since $b$ is disjoint from $\gamma,$ we get that $d_\mathcal{C}(\gamma, f_n^{g-1}(\gamma)) = 1,$ which implies that 
    $$l_\mathcal{C}(f_n) \le \frac{d_\mathcal{C}(\gamma, f_n^{g-1}(\gamma))}{g-1} = \frac{1}{g-1}.$$
\end{proof}

Lemma \ref{sf} shows that the stretch factors of the $f_n$ tend to infinity, while Lemma \ref{tl} implies that their stable curve graph translation length is bounded from above by the constant required. This finishes the proof of Theorem \ref{myothertheorem}.\\

Gadre and Tsai show that the stable curve graph translation length of $f_1,$ i.e. Penner's original example, can be bounded from above not just by $\frac{1}{g-1}$ but even by $\frac{4}{g^2+g-4}$ (see \cite{Vaibhav}, proof of Theorem 6.1). This is done by computing more iterates of the curve $\rho(b).$ In particular, Gadre and Tsai argue that an iterate of $\rho(b)$ is contained in a neighbourhood of a union of some of the curves in the set $\{\rho^i(a), \rho^i(b), \rho^i(c) \, | \, i=1,...,g\},$ and as long as this union is not over all the curves of the set, one can find a curve disjoint from both $\rho(b)$ and its iterate. So, the proof boils down to controlling the neighbourhoods for a high enough iterate. Since $f_n$ differs from $f_1$ only by twisting more often about $b,$ the iterates of $\rho(b)$ under $f_n$ are contained in the same neighbourhoods as the iterates of $\rho(b)$ under $f_1.$ Hence, looking at $f_n$ instead of $f_1$ does not affect Gadre and Tsai's argument, and we obtain $l_\mathcal{C}(f_n) \le \frac{4}{g^2 + g - 4}.$ We leave out the details, since our bound of $\frac{1}{g-1}$ is enough in order to prove Corollaries \ref{cor1} and \ref{cor2}.\\

\begin{customthm}{1.3}
    \textit{There exists a sequence $(h_g)_{g=2}^\infty$, where $h_g$ is a pseudo-Anosov of a genus $g$ surface, with
    $$l_{\mathcal{T}}(h_g) \to \infty \text{ and } l_\mathcal{C}(h_g) \to 0,$$ as $g \to \infty.$}
\end{customthm}
\begin{proof}
    For $g = 2,$ let $h_g$ be an arbitrary pseudo-Anosov. For $g \ge 3,$ let $n_g := g^g$ and let $h_g := f_{n_g},$ where $f_{n_g}$ is as constructed at the begining of Section \ref{pennerexample} for the corresponding genus $g.$\\

    From Lemma \ref{sf} we obtain $l_\mathcal{T}(h_g) \ge \log(g),$ and from Lemma \ref{tl} we get that $l_\mathcal{C}(f_g) \le \frac{1}{g-1}.$ It follows that 
    $$\lim\limits_{g \to \infty} l_\mathcal{T}(h_g) = \infty \text{ and } \lim\limits_{g \to \infty} l_\mathcal{C}(h_g) = 0.$$
\end{proof}

\begin{customthm}{1.4}
    \textit{For any $g \ge 3,$ there exists $q \in \mathbb{Q}$ such that there are infinitely many non-conjugate pseudo-Anosovs in $\Mod(S)$ with stable curve graph translation length $q.$}
\end{customthm}
\begin{proof}
    Let $L = \{l_\mathcal{C}(f) \, | \, f \in \Mod(S) \text{ pseudo-Anosov}\}$ be the set of all stable curve graph translation lengths of pseudo-Anosovs on $S.$ Bowditch proves that there is a power $m$ such that every pseudo-Anosov $f \in \Mod(S)$ preserves a geodesic in the curve graph after being raised to the power $m$ (see \cite{Bowditch}, Theorem 1.4). Hence, the set $L$ is contained in the set $\{\frac{n}{m} \, | \, n \in \mathbb{N}\}.$ In particular, this implies that for any constant $C,$ the set $\{x \in L \, | \, x \le C\}$ is finite.\\

    From Theorem \ref{myothertheorem}, we know that there is a sequence of pseudo-Anosovs $f_n$ with $\lambda(f_n) \longrightarrow \infty$ as $n \to \infty,$ and furthermore $l_\mathcal{C}(f_n) \le C$ for all $n,$ where $C$ can be taken to be any constant above $\frac{1}{g-1}.$ After having to possibly pass to a subsequence, we can assume that all the $\lambda(f_n)$ are pairwise different, and so all the $f_n$ are pairwise non-conjugate.\\ 
    
    Finally, since the set $\{l_{\mathcal{C}}(f_n) \, | \, n \in \mathbb{N}\}$ is a finite set by the discussion above, there must be a $q$ that is attained by infinitely many of the $f_n.$
\end{proof}

\section{Construction of infinite multiplicity}\label{end}

In this section, we prove Theorem \ref{construction of inf mult}. The proof relies on results about subsurface projections by Masur--Minsky (\cite{MM2}). We start by recalling some of the theory.\\

Let $S$ be a closed genus $g$ surface and $Y$ an incompressible, connected subsurface. For a set $X,$ let $\mathcal{P}(X)$ denote the set of finite subsets of $X.$ Masur--Minsky construct subsurface projection maps $\pi_Y: \mathcal{C}(S) \to \mathcal{P}(\mathcal{C}(Y)).$ The way to think about these maps is the following: Given a curve in $\mathcal{C}(S),$ look at its intersection with $Y.$ This intersection is a collection of arcs. Turn these arcs into closed curves by piecing them together with the boundary components of $Y.$ This yields a finite union of curves which we define as the image of $\pi_Y$ of the curve we started with (see \cite{MM2}, Section 2 for details). Note that the image of a curve is the empty set if and only if its intersection with $Y$ is empty.\\

We write $d_Y$ for the distance in $\mathcal{C}(Y).$ This should not be confused with $d_\mathcal{C}$ which as before denotes the distance in $\mathcal{C}(S).$ The crucial result of Masur--Minsky for this work is (\cite{MM2}, Theorem 3.1). We state it here for completeness.

\begin{theorem*}
    Let $Y$ be a subsurface of $S$ as above. Let $G$ be a geodesic in $\mathcal{C}(S)$ such that every vertex of $G$ has non-empty subsurface projection to $Y$. There exists a constant $M$ depending only on the genus of $S$ such that $\text{diam}(\pi_Y(G)) \le M.$
\end{theorem*}

In other words, the above theorem says that every geodesic in $\mathcal{C}(S),$ whose vertices have non-trivial image under $\pi_Y,$ projects to a uniformly bounded set. We'll make use of this as follows: Assume we are given curves $\beta_1, \beta_2 \in \mathcal{C}(S)$ with $d_Y(\pi_Y(\beta_1), \pi_Y(\beta_2)) > M.$ Then, any geodesic joining $\beta_1$ and $\beta_2$ in $\mathcal{C}(S)$ has to pass through a curve that doesn't intersect the subsurface $Y.$ In particular, we will apply the above theorem to $Y = S \setminus \alpha,$ where $\alpha$ is the curve from Theorem \ref{construction of inf mult}. Here, we make use of the fact that $\alpha$ is non-separating, so that $S \setminus \alpha$ is connected. Note that for $Y = S \setminus \alpha,$ the only curve in $S$ with empty subsurface projection to $Y$ is the curve $\alpha.$ Hence, in this case we can conclude that the geodesic connecting $\beta_1$ and $\beta_2$ has to pass through $\alpha.$\\ 

From now on, we omit to write the subsurface projection $\pi_Y$ when measuring the distance in $\mathcal{C}(Y)$ i.e. if $\beta_1, \beta_2$ are in $\mathcal{C}(S),$ then we simply write $d_Y(\beta_1, \beta_2)$ instead of $d_Y(\pi_Y(\beta_1), \pi_Y(\beta_2)).$\\ 

We recall the setup of Theorem \ref{construction of inf mult}. Let $f \in \Mod(S)$ be a pseudo-Anosov and assume there are curves $\alpha_i, \, i \in \mathbb{Z},$ such that $f(\alpha_i) = \alpha_{i+1}$ and $d_\mathcal{C}(\alpha_i, \alpha_j) = |i-j| \, m$ for all $i,j \in \mathbb{Z}$ and some fixed $m \in \mathbb{N}.$ Note that this ensures that $l_\mathcal{C}(f) = m.$ Furthermore, note that the condition on $\alpha$ in Theorem \ref{construction of inf mult} guaranties the existence of the $\alpha_i$ as the iterates of $\alpha$ under $f.$ We use the convention that $\alpha = \alpha_1.$\\

From now on, let $Y = S \setminus \alpha$ and $h$ be a pseudo-Anosov of $Y.$ Since $h$ is pseudo-Anosov, we know that $d_Y(\beta, h^k(\beta)) \to \infty$ as $k \to \infty$ for any curve $\beta$ in $Y.$ Choose $N_h \in \mathbb{N}$ big enough such that for all $n \ge N_h$ we have 
\begin{equation}\label{inequality}
    d_Y(\alpha_2, h^n(\alpha_2)) > 2M + d_Y(\alpha_0, \alpha_2),
\end{equation} where $M$ denotes the constant from the Masur-Minsky theorem. The reason for this choice will become apparent in the proof of Lemma \ref{distance lemma} below.\\ 

Fix some $n \ge N_h$ and define $ F = f_{h,n} := h^n \circ f.$ By abuse of notation, from now on we simply write $h$ instead of $h^n.$ Hence, we have $F = h \circ f$ and our goal is to compute $l_\mathcal{C}(F).$ \\ 

In order to compute the stable curve graph translation length of $F,$ we need to iterate $F$ on some curve. We choose the curve $\alpha_0.$ In the following, we use the notation $h_i := f^{i-1}hf^{-(i-1)}$ for all $i \ge 1.$ In particular, $h_1 = h.$ We obtain:

$$F(\alpha_0) = h_1 \circ f(\alpha_0) = \alpha_1,$$ as well as $$F^2(\alpha_0) = F(\alpha_1) = h_1(\alpha_2).$$ 
Iterating this process yields
$$F^k(\alpha_0) = h_1 \, f \, ... \, h_1 \, f \, (\alpha_0) = h_1 \, \underbrace{f \, h_1 \, f^{-1}}_{h_2} \, \underbrace{f^2 \, h_1 \, f^{-2}}_{h_3} \, f^3 \, h_1 \, ... \, \underbrace{f^{k-1} \, h_1 \, f^{-(k-1)}}_{h_{k-1}} \, f^k \, (\alpha_0) = $$ $$ = h_1 \, h_2 \, h_3 \, ... \, h_{k-1} \, (\alpha_k).$$

Hence, to prove Theorem \ref{construction of inf mult}, we have to compute the distance in the curve graph between $\alpha_0$ and $h_1 \, h_2 \, h_3 \, ... \, h_{k-1} \, (\alpha_k).$ We formulate this as a Lemma \ref{distance lemma} below. The following Figure \ref{fig:twists in curve graph} is a schematic drawing of the vertices of the curve graph we are interested in and should help the reader visualize the distances appearing in the proof of Lemma \ref{distance lemma}. Note that $f$ "acts on" Figure \ref{fig:twists in curve graph} by translating horizontally.

\begin{figure}[h]
    \centering
    \begin{tikzpicture}
    
        \node[anchor=south west,inner sep=0] at (0,0){\includegraphics[scale=0.7]{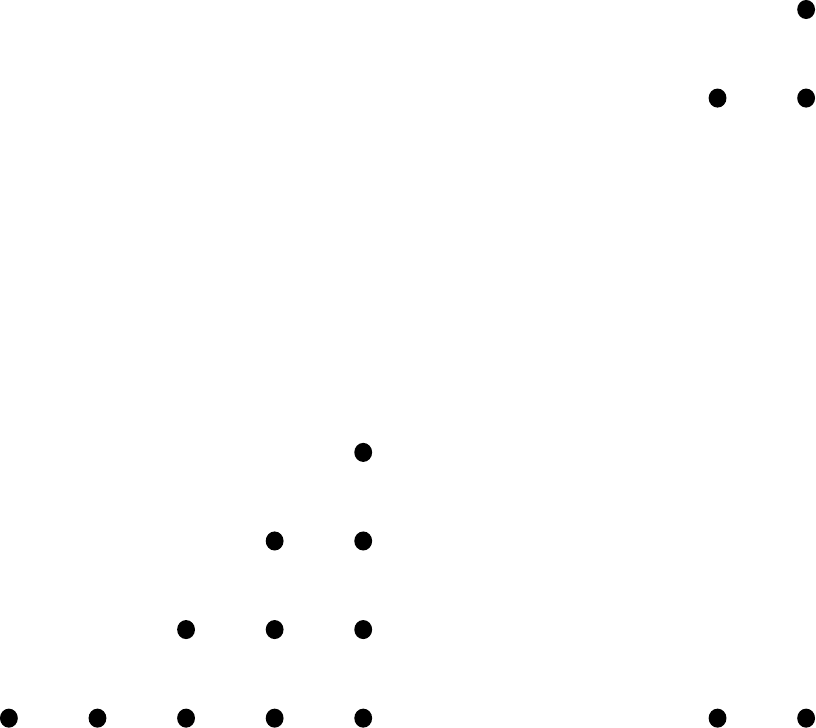}};
        \node at (0.2,-0.3) {$\alpha_0$};
        \node at (1.25, -0.3) {$\alpha_1$};
        \node at (2.3, -0.3) {$\alpha_2$};
        \node at (3.35, -0.3) {$\alpha_3$};
        \node at (4.4, -0.3) {$\alpha_4$};
        \node at (8.6, -0.3) {$\alpha_K$};
        \node at (9.65, -0.3) {$\alpha_{K+1}$};
        \node at (2, 0.75) {$h_1(\alpha_2)$};
        \node at (3.35, 0.75) {$h_2(\alpha_3)$};
        \node at (4.7, 0.75) {$h_3(\alpha_4)$};
        \node at (2.95, 1.8) {$h_1h_2(\alpha_3)$};
        \node at (4.8, 1.8) {$h_2h_3(\alpha_4)$};
        \node at (4.4, 2.85) {$h_1h_2h_3(\alpha_4)$};
        \node at (7, 7.45) {$h_1...h_{K-1}(\alpha_K)$};
        \node at (11.1, 7.45) {$h_2...h_K(\alpha_{K+1})$};
        \node at (11.1, 8.5) {$h_1...h_K(\alpha_{K+1})$};
        \node at (6.5, 0.1) {$\cdots$};
        \node at (8.6, 3.675) {$\vdots$};
        \node at (9.65, 3.675) {$\vdots$};
        \node at (6.4, 5.25) {$.$};
        \node at (6.5, 5.35) {$.$};
        \node at (6.6, 5.45) {$.$};
        
    \end{tikzpicture}
    \caption{Schematic diagram of part of the curve graph}
    \label{fig:twists in curve graph}
\end{figure}

\begin{lemma}\label{distance lemma} 
    It holds that $$d_\mathcal{C}(\alpha_0, \, h_1 \, h_2 \, h_3 \, ... \, h_{k-1} \, (\alpha_k)) = km,$$ for any $k \in \mathbb{N}.$
\end{lemma}
\begin{proof}
    We prove this statement by induction over $k.$ The base case $k = 1$ holds since $d_\mathcal{C}(\alpha_0, \alpha_1) = m$ by assumption.\\ 
    
    We also do the case $k = 2$ explicitly since it gives an insight to how the induction step works. For this case, we want to compute $d_\mathcal{C}(\alpha_0, h_1(\alpha_2)).$ Since $d_\mathcal{C}(\alpha_0, \alpha_1) = m$ and $$d_\mathcal{C}(\alpha_1, h_1(\alpha_2)) = d_\mathcal{C}(h_1(\alpha_1), h_1(\alpha_2)) = d_\mathcal{C}(\alpha_1, \alpha_2) = m,$$ we obtain that $d_\mathcal{C}(\alpha_0, h_1(\alpha_2)) \le 2m.$ Furthermore, using equation (\ref{inequality}) we obtain 
    $$d_Y(\alpha_0, h_1(\alpha_2)) \, \ge \, d_Y(\alpha_2, h_1(\alpha_2)) - d_Y(\alpha_0, \alpha_2) > 2M > M.$$

    It follows from the Masur-Minsky theorem that any geodesic between $\alpha_0$ and $h_1(\alpha_2)$ in $\mathcal{C}(S)$ has to pass through the curve $\alpha_1.$ We obtain:
    $$d_\mathcal{C}(\alpha_0, h_1(\alpha_2)) \ge d_\mathcal{C}(\alpha_0, \alpha_1) + d_\mathcal{C}(\alpha_1, h_1(\alpha_2)) = m + m = 2m.$$
    Assume now that for all natural numbers $k$ up to some $K \in \mathbb{N},$  the claim
    $$d_\mathcal{C}(\alpha_0, h_1 \, h_2 \, h_3 \, ... \, h_{k-1} \, (\alpha_k)) = km$$
    holds. We have to show the analogous claim for $K+1.$\\

    Note that $f^{-1}(h_2 \, ... \, h_K(\alpha_{K+1})) = h_1 \, ... \, h_{K-1}(\alpha_K).$ It follows that
    $$d_\mathcal{C}(\alpha_1, h_1 \, ... \, h_K(\alpha_{K+1})) = d_\mathcal{C}(\alpha_1, h_2 \, ... \, h_K(\alpha_{K+1})) = d_\mathcal{C}(\alpha_0, h_1 \, ... \, h_{K-1}(\alpha_K)) = Km,$$
    and hence
    $$d_\mathcal{C}(\alpha_0, h_1 \, ... \, h_K(\alpha_{K+1})) \le d_\mathcal{C}(\alpha_0, \alpha_1) + d_\mathcal{C}(\alpha_1, h_1 \, ... \, h_K(\alpha_{K+1})) \le m + Km = (K+1)m.$$
    For the other inequality, we want to use the subsurface projection to $Y.$ Our goal is to show that the distance of the projections of $\alpha_0$ and $h_1 \, ... \, h_K(\alpha_{K+1})$ is sufficiently big.\\

    First, note that 
    $$d_\mathcal{C}(h_1(\alpha_2), \, h_1 \, ... \, h_K(\alpha_{K+1})) = d_\mathcal{C}(\alpha_2, \, h_2 \, ... \, h_K(\alpha_{K+1})),$$
    and by applying $f^{-1}$ to both sides the above is equal to 
    $$d_\mathcal{C}(\alpha_1, \, h_1 \, ... \, h_{K-1}(\alpha_K)) = d_\mathcal{C}(\alpha_1, \, h_2 \, ... \, h_{K-1}(\alpha_K)),$$
    which by applying $f^{-1}$ again is equal to
    $$d_\mathcal{C}(\alpha_0, \, h_1 \, ... \, h_{K-2}(\alpha_{K-1})),$$
    which by the induction hypothesis equals $(K-1)m.$ From the calculations above, we also know that
    $$d_\mathcal{C}(\alpha_1, \, h_1 \, ... \, h_K(\alpha_{K+1}))$$ equals $Km.$\\

    We now want to argue that a geodesic between $h_1(\alpha_2)$ and $h_1 \, ... \, h_K(\alpha_{K+1})$ doesn't pass through $\alpha_1.$ So, assume such a geodesic did pass through $\alpha_1.$ Then:
    $$d_\mathcal{C}(h_1(\alpha_2), \, h_1 \, ... \, h_K(\alpha_{K+1})) \ge d_\mathcal{C}(h_1(\alpha_2), \alpha_1) \,   + \,  d_\mathcal{C}(\alpha_1, h_1 \, ... \, h_K(\alpha_{K+1})) \,  =$$ 
    $$= \,  m \, + \, Km \, = (K+1)m.$$

    But then $(K-1)m \, = \, d_\mathcal{C}(h_1(\alpha_2), \, h_1 \, ... \, h_K(\alpha_{K+1})) \, \ge \, (K+1)m$ which is a contradiction. Hence, no geodesic between $h_1(\alpha_2)$ and $h_1 \, ... \, h_K(\alpha_{K+1})$ passes through $\alpha_1.$ In particular, this means that any such geodesic consists of vertices with non-empty subsurface projection image to $Y.$ So, we can use the Theorem of Masur--Minsky to conclude that a projection to $\mathcal{C}(Y)$ of such a geodesic has image whose diameter is bounded by $M.$ In particular, we obtain
    $$d_Y(h_1(\alpha_2), \, h_1 \, ... \, h_K(\alpha_{K+1})) \le M.$$

    Using equation (\ref{inequality}) again, it follows that
    $$d_Y(\alpha_2, h_1 \, ... \, h_K(\alpha_{K+1})) \ge d_Y(\alpha_2, h_1(\alpha_2)) \, - d_Y(h_1 \, ... \, h_K(\alpha_{K+1}), \, h_1(\alpha_2)) \, >$$
    $$> \, 2M + d_Y(\alpha_0, \alpha_2) \, - \, M  \, = \, M \, + \, d_Y(\alpha_0, \alpha_2),$$

    and finally,
    $$d_Y(\alpha_0, h_1 \, ... \, h_K(\alpha_{K+1})) \ge d_Y(\alpha_2, h_1 \, ... \, h_K(\alpha_{K+1})) - d_Y(\alpha_0, \alpha_2) >$$ 
    $$ > M + d_Y(\alpha_0, \alpha_2) - d_Y(\alpha_0, \alpha_2) = M.$$

    This inequality together with the Masur-Minsky theorem shows that any geodesic from $\alpha_0$ to $h_1 \, ... \, h_K(\alpha_{K+1})$ has to pass through $\alpha_1.$ For the distance of the two points, this implies:
    $$d_\mathcal{C}(\alpha_0, h_1 \, ... \, h_K(\alpha_{K+1})) \ge d_\mathcal{C}(\alpha_0, \alpha_1) + d_\mathcal{C}(\alpha_1, h_1 \, ... \, h_K(\alpha_{K+1})) = $$
    $$= m + Km = (K+1)m.$$
\end{proof}

We are now ready to prove Theorem \ref{construction of inf mult}:

\begin{proof}(Theorem \ref{construction of inf mult})\\
    
    It only remains to show that $F = h \circ f$ satisfies $l_\mathcal{C}(F) = l_\mathcal{C}(f) = m.$ By definition of the stable curve graph translation length, we have
    $$l_\mathcal{C}(F) = \liminf\limits_{k \to \infty} \frac{d_\mathcal{C}(\alpha_0, F^k(\alpha_0))}{k}.$$
    Above, we showed that 
    $$F^k(\alpha_0) = h_1 \, ... \, h_{k-1}(\alpha_k),$$
    and in Lemma \ref{distance lemma} we computed that
    $$d_\mathcal{C}(\alpha_0, h_1 \, ... \, h_{k-1}(\alpha_k)) = km.$$
    Hence, we conclude
    $$l_\mathcal{C}(F) = \liminf\limits_{k \to \infty} \frac{km}{k} = m.$$  
\end{proof}

\begin{remark}
    For a pseudo-Anosov $f$ that satisfies the condition of Theorem \ref{construction of inf mult}, we now have a way of constructing infinitely many pseudo-Anosovs different from $f$ with the same stable curve graph translation length. Note that we can even change the curve $\alpha = \alpha_1$ and use any other curve of the $\alpha_k$ in the construction. More precisely, for any curve $\alpha_k$ and any pseudo-Anosov $h \in \Mod(S \setminus \alpha_k),$ there are infinitely many pseudo-Anosovs of the form $h^n \circ f$ for big enough $n$ that satisfy $l_\mathcal{C}(h^n \circ f) = l_\mathcal{C}(f).$
\end{remark}

The remark talks about how to obtain pseudo-Anosovs not equal to $f$ with the same stable curve graph translation length. However, this would not be as interesting if all the newly obtained pseudo-Anosovs were conjugate to each other. In the following lemma, we ensure that this doesn't happen if we add an extra assumption to $f$.  

\begin{lemma}
    Assume that there is a symplectic basis of homology of the form $\alpha, \beta_1, ... , \beta_{2g-1},$ where $\alpha$ is any of the $\alpha_k,$ such that the image under the symplectic representation of $f$ with respect to this basis is a primitive matrix $M.$ Then, there is a pseudo-Anosov $h$ of $S \setminus \alpha$ such that the sequence $h^n \circ f$ contains infinitely many non-conjugate mapping classes.
\end{lemma}
\begin{proof}
    Let $\alpha, \beta_1, ... , \beta_{2g-1}$ be a symplectic basis of $H_1(S),$ where $\alpha$ and $\beta_1$ intersect exactly once. Let $M$ be the matrix defined by action of $f$ on homology with respect to the above basis, and assume that $M$ is primitive.\\
    
    Let $S \setminus \alpha$ be the surface with two boundary components that results by cutting $S$ at $\alpha.$ Let $S'$ be the closed surface of genus $g-1$ that results from collapsing the two boundary components of $S \setminus \alpha$ to points $p_1, p_2$ respectively. Note that $\beta_1, ... , \beta_{2g-1}$ is a basis of $H_1(S', \{p_1, p_2\}),$ where $\beta_1$ is now an arc connecting the two points $p_1$ and $p_2.$ A homeomorphism fixing the boundary of $S \setminus \alpha$ is pseudo-Anosov if and only if the corresponding homeomorphism of $S' \setminus \{p_1, p_2\}$ is pseudo-Anosov. Choose $h$ such that the action of $h$ on $H_1(S', \{p_1, p_2\})$ with respect to the basis $\beta_1, ... , \beta_{2g-1}$ is given by a primitive matrix. Call this matrix $A.$\\

    Then, seen as a mapping class of $S,$ $h$ acts on $H_1(S)$ by the matrix 
    $$\begin{pmatrix}
        1 & b\\
        0 & A
    \end{pmatrix},$$ where $b$ is a row vector with $2g-1$ entries and $0$ stands for a column vector with $2g-1$ entries. After possibly pre-composing $h$ with Dehn twists about $\alpha,$ we can assume that $b$ is a non-zero vector consisting of only non-negative entries. Note that this corresponds to changing $h$ by twists about the boundary components of $S \setminus \alpha$ which doesn't change the fact that $h$ is a pseudo-Anosov of $S \setminus \alpha$.\\

    Computing powers of this matrix shows that $h^n$ acts as 
    $$\begin{pmatrix}
        1 & b(A+A^2+...+A^{n-1})\\
        0 & A^n
    \end{pmatrix}.$$

    We obtain that $h^n \circ f$ acts as
    $$\begin{pmatrix}
        1 & b(A+A^2+...+A^{n-1})\\
        0 & A^n
    \end{pmatrix}M.$$

    Write 
    $$M = \begin{pmatrix}
        m_{11} & m_r\\
        m_c & M'
    \end{pmatrix},$$
    where $m_r$ and $m_c$ are a row- and a column vector respectively of dimension $2g-1,$ and $M'$ is a $(2g-1) \times (2g-1)$ matrix, and compute:

    $$\begin{pmatrix}
        1 & b(A+A^2+...+A^{n-1})\\
        0 & A^n
    \end{pmatrix} \begin{pmatrix}
        m_{11} & m_r\\
        m_c & M'
    \end{pmatrix} = $$\\ $$ = \begin{pmatrix}
        m_{11} + b(A+A^2+...+A^{n-1})m_r & m_r+b(A+A^2+...+A^{n-1})M'\\
        A^nm_c & A^nM'
    \end{pmatrix}.$$

    Since $A$ and $M$ are both primitive, so is the above matrix. Letting $n$ go to infinity, one sees that all the entries of the above matrix tend to infinity, which implies that the spectral radii of the matrices representing $h^n \circ f$ tend to infinity. Since conjugate matrices have the same eigenvalues, this shows that there is an infinite subsequence $n_k$ such that all the $h^{n_k} \circ f$ are pairwise non-conjugate. 
\end{proof}

The extra assumption about the homology action of $f$ in the above lemma excludes for example pseudo-Anosovs in the Torelli group. However, we believe that the construction in Theorem \ref{construction of inf mult} can be used to construct infinitely many non-conjugate pseudo-Anosovs even without the homology assumption. Note that the remark above says that we can use any triple $(\alpha_k, h, n),$ where $h$ is a pseudo-Anosov of $S \setminus \alpha_k$ and $n$ a sufficiently high power, to build the suitable pseudo-Anosovs $h^n \circ f.$ In the lemma, by adding the homology assumption on $f,$ we found infinitely many non-conjugate pseudo-Anosovs only by varying $n.$ Hence, by varying $\alpha_k$ and $h$ as well, one might be able to find infinitely-many non-conjugate pseudo-Anosovs without having to impose the extra assumption. A thorough study of the conjugacy classes of the maps $h^n \circ f$ that depend on the triple $(\alpha_k, h, n)$ would be interesting.

\printbibliography

\end{document}